\documentclass[twoside,a4paper]{elsarticle}

\usepackage{amsmath,amsthm,amssymb,enumerate}
\usepackage{epsfig}
\usepackage{amssymb}
\usepackage{amsmath}
\usepackage{amssymb}
\usepackage{amsmath,amsthm}
\usepackage[latin1]{inputenc}
\usepackage[T1]{fontenc}
\usepackage{path}
\usepackage{ae,aecompl}
\usepackage{amsfonts}
\usepackage{amsxtra}
\usepackage{bbm,euscript,mathrsfs}

\allowdisplaybreaks

\usepackage{enumitem}
\setenumerate{label={\rm (\alph{*})}}

\usepackage{amsfonts}
\usepackage{amsxtra}

\usepackage[frak,hyperref,theorem]{paper_diening}

\newcommand{\Div}{\divergence}
\newcommand{\ep}{\bfvarepsilon}
\newcommand{\R}{\mathbb R}
\newcommand{\N}{\mathbb N}
\newcommand{\s}{\mathbb S}
\newcommand{\E}{\mathbb E}
\newcommand{\p}{\mathbb P}
\newcommand{\Q}{\mathcal Q}
\newcommand{\F}{\mathcal F}
\newcommand{\LL}{\mathcal L}

\newcommand{\dd}{\mathrm{d}}
\newcommand{\dx}{\,\mathrm{d}x}
\newcommand{\dt}{\,\mathrm{d}t}
\newcommand{\dxt}{\,\mathrm{d}x\,\mathrm{d}t}
\newcommand{\ds}{\,\mathrm{d}\sigma}
\newcommand{\dxs}{\,\mathrm{d}x\,\mathrm{d}\sigma}

\DeclareMathOperator{\Bog}{Bog}
\DeclareMathOperator{\Var}{Var}

\begin{document}

 \begin{center}
   \Large Existence theory for stochastic power law fluids\mdseries \\[2ex]
   \large Dominic Breit\quad 2015
\normalsize\\[3ex]

Department of Mathematics, Heriot-Watt University\\ 
Edinburgh EH14 4AS, UK\\
Mathematical Institute, LMU Munich\\ Theresienstra\ss e 39, 80333 Munich, Germany

 	\end{center}

\textbf{abstract:}
We consider the equations of motion for an incompressible Non-Newtonian fluid in a bounded Lipschitz domain $G\subset\R^d$
during the time interval $(0,T)$ together with a stochastic perturbation driven by a Brownian motion $\bfW$. The balance of momentum reads as
$$\dd\bfv=\Div\bfS\dt-(\nabla\bfv)\bfv\dt+\nabla\pi\dt+\bff\dt+\Phi(\bfv)\,\dd\bfW_t,$$
where $\bfv$ is the velocity, $\pi$ the pressure and $\bff$ an external volume force. We assume the common power law model $\bfS(\ep(\bfv))=\big(1+|\ep(\bfv)|\big)^{p-2} \ep(\bfv)$ and show the existence of weak (martingale) solutions provided $p>\tfrac{2d+2}{d+2}$. Our approach is based on the $L^\infty$-truncation and a harmonic pressure decomposition which are adapted to the stochastic setting.\\\

  \textbf{keyword:}
Non-Newtonian fluids, weak solution, martingale solution, stochastic parabolic PDE's, generalized Navier-Stokes equations, $L^\infty$-truncation, pressure decomposition


\section{Introduction}

The flow of a homogeneous incompressible fluid in a bounded Lipschitz body $G\subset\R^d$, ($d=2,3$), during the time interval $(0,T)$ is described by the following set of equations on $\Q:=(0,T)\times G$ (see for instance \cite{BAH})\footnote{For a better understanding of the problem we start with a survey about the deterministic problem but the topic of the paper is the corresponding SPDE.}
\begin{align}\label{1.1}
\left\{\begin{array}{rc}
-\rho\partial_t \bfv+\Div \bfS=\rho(\nabla \bfv)\bfv+\nabla \pi-\rho\bff& \mbox{in $\Q$,}\\
\Div \bfv=0\qquad\qquad\,\,\,\,& \mbox{in $\Q$,}\\
\bfv=0\qquad\,\,\,\,\quad\quad& \mbox{ \,on $\partial G$,}\\
\bfv(0,\cdot)=\bfv_0\,\qquad\qquad&\mbox{ \,in $G$.}\end{array}\right.
\end{align}
Here the unknown quantities are the velocity field $\bfv:\Q\rightarrow\R^d$, $\Q:=(0,T)\times G$, and the pressure $\pi:\Q\rightarrow\R$. The functions $\bff:\Q\rightarrow \R^d$ represent a system of volume forces, $\bfv_0:G\rightarrow\R^d$ the initial datum, $\bfS:\Q\rightarrow\s^d$ is the stress deviator\footnote{$\s^d$:=space of all symmetric $d\times d$ matrices; the full stress tensor is $\bfsigma=\bfS-\pi \bfI$.} and $\rho>0$ the density of the fluid. Equation (\ref{1.1}$)_1$  and (\ref{1.1}$)_2$ describe the conservation of balance and the conservation of mass respectively. Both are valid for all homogeneous liquids and gases. In order to describe a specific fluid one needs a constitutive law which relates the stress deviator $\bfS$ to the symmetric gradient $\ep(\bfv):=\tfrac{1}{2}\big(\nabla\bfv+\nabla\bfv^T\big)$ of the velocity $\bfv$. In the easiest case this relation is linear, i.e.,
\begin{align}\label{1.2}
\bfS=\bfS(\ep(\bfv))=\nu \ep(\bfv),
\end{align} 
where $\nu>0$ is the viscosity of the fluid. In this case we have $\Div\bfS=\nu \Delta\bfv$ and (\ref{1.1}) are the classical Navier-Stokes equations. Its mathematical observation started with the work of Leray and Ladyshenskaya (see \cite{La} and for a more recent approach \cite{Ga1,Ga2}). The existence of a weak solution (where derivatives are to be understood in a distributional sense) can be shown by nowadays standard arguments. However the regularity (i.e. the existence of a strong solution) is still open.\\
Only fluids with simple molecular structure e.g. water, oil and certain gases fulfil a linear relation such as (\ref{1.2}). Those who does not are called Non-Newtonian fluids (see \cite{AM}). A special class among these are generalized Newtonian fluids. Here the viscosity is assumed to be a function of the shear rate $|\ep(\bfv)|$  and the constitutive relations reads as
\begin{align}\label{1.3}
\bfS(\ep(\bfv))=\nu(|\ep(\bfv)|) \ep(\bfv).
\end{align} 
An external force can produce two different reactions:
\begin{itemize}
\item The fluid becomes thicker (for example batter): the viscosity of a shear thickening fluid is an increasing function of the shear rate;
\item The fluid becomes thinner (for example ketchup): the viscosity of a shear thinning fluid is a decreasing function of the shear rate.
\end{itemize}
The power-law model for Non-Newtonian respectively generalized Newtonian fluids
\begin{align}\label{1.4'}
\bfS(\ep(\bfv))=\nu_0\big(1+|\ep(\bfv)|\big)^{p-2} \ep(\bfv)
\end{align}
is very popular among rheologists.
Here $\nu_0>0$ and $p\in(1,\infty)$ are specified by physical experiments. An extensive list for specific $p$-values for different fluids can be found in \cite{BAH}. Apparently many interesting $p$-values lie in the interval $[\frac 32,2]$.\\
The mathematical discussion of power-law models started in the late sixties with the work of Lions and Ladyshenskaya (see \cite{La}-\cite{La3} and \cite{Li}).
Due to the appearance of the convective term the equations for power law fluids (the constitutive law is given by (\ref{1.4'})) significantly depend on the value of $p$. The first results were achieved by Ladyshenskaya and Lions for $p\geq\frac{3d+2}{d+2}$ (see \cite{La} and \cite{Li}). They show the existence of a weak solution in the space
\begin{align*}
L^p(0,T;W^{1,p}_{0,\Div}(G))\cap L^\infty(0,T;L^2(G)).
\end{align*}
In this case it follows from parabolic interpolation that $\bfv\otimes\bfv:\ep(\bfv)\in L^1(\Q)$. So the solution is also a test-function and the existence proof is based on monotone operator theory and compactness arguments. \\
This results were improved by Wolf \cite{Wo} to the case $p>\frac{2d+2}{d+2}$ via $L^\infty$-truncation. In this situation we have that $(\nabla\bfv)\bfv\in L^1(\Q)$ and therefore we can test with functions from $L^\infty(\Q)$.
 The basic idea (which was already used in the stationary case in \cite{FrMaSt} together with the bound $p\geq\frac{2d}{d+1}$) is to approximate $\bfv$ by a bounded function $\bfv_\lambda$ which is equal to $\bfv$ on a large set and its $L^\infty$-norm can be controlled by $\lambda$.\\
Wolf's result was improved to $p>\frac{2d}{d+2}$ in \cite{DRW} and \cite{BrDS} by the Lipschitz truncation method. Under this restriction to $p$ we have $\bfv\otimes\bfv\in L^1(\Q)$ which means we can test by Lipschitz continuous functions. So one has to approximate $\bfv$ by a Lipschitz continuous function $\bfv_\lambda$ which is quite challenging in the parabolic situation.\footnote{The easier steady case was observed in \cite{FMS2} and \cite{DMS}}\\
From several points of view it is reasonable to add a stochastic part to the equation of motion. 
\begin{itemize}
\item It can be understood as a turbulence in the fluid motion (see \cite{mikul}). 
\item It can be interpreted as a perturbation from the physical model. 
\item Apart from the force $\bff$ we are observing there might be further quantities with a (usually small) influence on the motion.
\end{itemize}
We are therefore interested in the set of equations:\footnote{We neglect physical constants for simplicity.}
\begin{align}\label{eq:}
\left\{\begin{array}{rc}
\dd\bfv=\Div\bfS\dt-(\nabla\bfv)\bfv\dt+\nabla\pi\dt+\bff\dt+\Phi(\bfv)\dd\bfW_t
& \mbox{in $\Q$,}\\
\Div \bfv=0\qquad\qquad\qquad\qquad\qquad\,\,\,\,& \mbox{in $\Q$,}\\
\bfv=0\qquad\qquad\qquad\qquad\,\,\,\,\quad\quad& \mbox{ \,on $\partial G$,}\\
\bfv(0)=\bfv_0\,\qquad\qquad\qquad\qquad\qquad&\mbox{ \,in $G$,}\end{array}\right.
\end{align}
with $\bfS$ given by (\ref{1.4'}).
We assume that $\bfW$ is a Brownian motion with values in a Hilbert space (see section 2 for details) .
We suppose that $\Phi$ growths linearly - roughly speaking $|\Phi(\bfv)|\leq c(1+|\bfv|)$ and $|\Phi'(\bfv)|\leq c$ (for a precise formulation see (\ref{eq:phi}) in section 2).
The idea behind this is an interaction between the solution and the random perturbation caused by
the Brownian motion. For large values of $|\bfv|$ we expect a larger perturbation than for small values.\\
There is a huge literature regarding the existence of weak solutions to the stochastic Navier-Stokes equations starting with the paper \cite{BeTe} by Bensoussan and Temam. For a recent overview we refer to \cite{Fl2}. However there seems to be a very limited knowledge about the Non-Newtonian fluid problem. In \cite{ChCh} a bipolar shear thinning fluid is observed. The authors of \cite{ChCh} assume the constitutive relation
\begin{align*}
\bfS=\nu_0\big(1+|\ep(\bfv)|\big)^{p-2} \ep(\bfv)-\nu_1\Delta\ep(\bfv),
\end{align*}
where $\nu_0,\nu_1>0$ and $1< p\leq 2$.
Compared with our model this results in an additional bi-Laplacian $\Delta^2\bfv$ in the equations of motion. This gives enough initial regularity to argue directly with monotone operators without using any form of truncation. Moreover, the main part of the equation is linear thus there is no problem with going to the limit in the approximated equation.\\
A further observation of stochastic power law fluids is done in \cite{Yo} and \cite{TeYo}. Following the approach in \cite{MNRR} they consider periodic boundary conditions and obtain existence for $p\geq\tfrac{9}{5}$ (in dimension 3). The restriction to a periodic boundary allows them to test the equation by 
the Laplacian of the solution (without using cut-off functions), which is not possible in general.
Moreover, in \cite{Yo} and \cite{TeYo} there is no interaction between the solution and the Brownian motion, modelled via the function $\Phi$ in (\ref{eq:}) which be quite reasonable also form the physical point of view.\\
We will investigate an existence theory which removes all this drawbacks. Our final result is the existence of a martingale weak solution to (\ref{eq:}) with $\bfS$ given by (\ref{1.4'}) in the sense of Definition \ref{def:weak} provided $p>\frac{2d+2}{d+2}$.
This solution is weak both in the analytical sense and in the probabilistic sense. The precise statement can be found in Theorem \ref{thm:main} in the next session. In fact, we extend the  results from \cite{Wo} to the stochastic fashion where the solution space is
$$L^2(\Omega,\F,\p;L^\infty(0,T;L^2(G)))\cap L^p(\Omega,\F,\p;L^p(0,T;W_{0,\Div}^{1,p}(G))).$$

Our procedure is as follows:  After a precise formulation of the stochastic background in section 2 we
investigate the pressure. As usual the pressure disappears in the weak formulation (see definition \ref{def:weak}) but can be reconstructed. Following the ideas from \cite{Wo} we relate to each term in the equation a pressure part. So also a stochastic part of the pressure is included. In section 4 we study auxiliary problems which are stabilized by adding a large power of $\bfv$. This approach is based on the Galerkin method.\\
In section 5 we prove the main theorem. Here we follow the approach in \cite{Wo} adapted to the stochastic fashion. The problems are as usual the convergences in the nonlinear parts of the approximated system. 
We
have to combine the techniques from nonlinear PDEs with stochastic calculus
for martingales. Note that it is not possible to work directly with test
functions. Instead of this we apply It\^{o}'s formula to certain functions of $\bfv$. Finally we use monotone operator theory combined with $L^\infty$-truncation to justify the limit procedure in the nonlinear tensor $\bfS$.

\section{Probability framework \& main theorem}
Let $(\Omega,\mathcal F,\p)$ be a probability space equipped with a filtration $\left\{\mathcal F_t,\,\,0\leq t\leq T\right\}$, which is a nondecreasing
family of sub-$\sigma$-fields of $\mathcal F$, i.e. $\mathcal F_s\subset\mathcal F_t$ for $0\leq s\leq t\leq T$. We further assume that $\left\{\mathcal F_t,\,\,0\leq t\leq T\right\}$ is right-continuous and $\mathcal F_0$ contains all the $\p$-negligible events in $\mathcal F$.\\
For a Banach space $(X,\|\cdot\|_X)$ we denote by for $1\leq p<\infty$ by $L^p(\Omega,\mathcal F,\p;X)$ the Banach space of all measurable function $v:\Omega\rightarrow X$ such that
\begin{align*}
\E\big[\|v\|_X^p\big]<\infty,
\end{align*}
where the expectation is taken w.r.t. $(\Omega,\mathcal F,\p)$.\\
Let $U$ be a Hilbert space with orthonormal basis $(\bfe_k)_{k\in\N}$ and let $L_2(U,L^2(G))$ be the set of Hilbert-Schmidt operators from $U$ to $L^2(G)$.  Define further the auxiliary space $U_0\supset U$  as
\begin{align}\label{eq:U0}
\begin{aligned}
U_0&:=\left\{\bfe=\sum_k \alpha_k\bfe_k:\,\,\sum_k \frac{\alpha_k^2}{k^2}<\infty\right\},\\
\|\bfe\|^2_{U_0}&:=\sum_{k=1}^\infty \frac{\alpha_k^2}{k^2},\quad \bfe=\sum_k \alpha_k\bfe_k.
\end{aligned}
\end{align}
Throughout the paper we consider a cylindrical Wiener process
$\bfW=(\bfW_t)_{t\in[0,T]}$ on $(\Omega,\F,(\F_t),\p)$ which has the form
\begin{align}\label{eq:W}
\bfW_\sigma=\sum_{k\in\N}\bfe_k \beta_k(\sigma)
\end{align}
with a sequence $(\beta_k)$ of independent real valued Brownian motions on $(\Omega,\mathcal F,(\F_t),\p)$. The embedding $U\hookrightarrow U_0$ is Hilbert-Schmidt and trajectories of $\bfW$ are $\p$-a.s. continuous (see \cite{PrZa}) with values in $U_0$.
Now 
\begin{align*}
\int_0^t \psi(\sigma)\,\dd\bfW_\sigma,\quad \psi\in L^2(\Omega,\F,\p;L^2(0,T;L_2(U,L^2(G)))),
\end{align*}
$\psi$ progressively measurable,
defines a $\p$-almost surely continuous $L^2(\Omega)$ valued $\mathcal F_t$-martingale.\footnote{for stochastic calculus in infinite dimensions we refer to \cite{PrZa}} Moreover, we can multiply with test-functions since
 \begin{align*}
\int_G\int_0^t \psi(\sigma)\,\dd\bfW_\sigma\cdot \bfphi\dx=\sum_{k=1}^\infty \int_0^t\int_G \psi(\sigma)( \bfe_k)\cdot\bfphi\dx\,\dd\beta_k(\sigma),\quad \bfphi\in L^2(G),
\end{align*}
is well-defined.\\
We suppose the following linear growth assumptions on $\Phi$ (following \cite{Ho1}): For each $\bfz\in L^2(G)$ there is a mapping $\Phi(\bfz):U\rightarrow L^{2}(G)$ defined by $\Phi(\bfz)\bfe_k=g_k(\bfz(\cdot))$. In particular, we suppose
that $g_k\in C^1(\R^d)$
and the following conditions for some $L\geq0$
\begin{align}\label{eq:phi}
&\sum_{k\in\N}|g_k(\bfxi)| \leq L(1+|\bfxi|),\quad\sum_{k\in\N}|\nabla g_k(\bfxi)|^2 \leq L,\quad\bfxi\in\R^d.
\end{align}
Not that the first assumption in (\ref{eq:phi}) is slightly stronger than 
\begin{align*}
&\sum_{k\in\N}|g_k(\bfxi)|^2 \leq L(1+|\bfxi|^2),\quad\bfxi\in\R^d,
\end{align*}
supposed in \cite{Ho1} and additionally implies
\begin{align}\label{eq:phi2}
\sup_{k\in\N} k^2|g_k(\bfxi)|^2\leq c(1+|\bfxi|^2).
\end{align}
Now we are ready to give a precise formulation of the meaning of solutions.
\begin{definition}[Solution]\label{def:weak}
Let $\Lambda_0,\Lambda_\bff$ be Borel probability measures on $L^2_{\Div}(G)$ and $L^2(\Q)$ respectively. Then
$$\big((\Omega,\mathcal F,(\mathcal F_t),\p),\bfv,\bfv_0,\bff,\bfW)$$
is called a martingale weak solution to \eqref{eq:} with $\bfS$ given by (\ref{1.4'}) with the initial datum $\Lambda_0$ and right-hand-side $\Lambda_\bff$ provided
\begin{enumerate}
\item $(\Omega,\mathcal F,(\mathcal F_t),\p)$ is a stochastic basis with a complete right-continuous filtration,
\item $\bfW$ is an $(\mathcal F_t)$-cylindrical Wiener process,
\item $\bfv\in L^2(\Omega,\F,\p;L^\infty(0,T;L^2(G)))\cap L^p(\Omega,\F,\p;L^p(0,T;W^{1,p}_{0,\Div}(G)))$ is progressively measurable,
\item $\bfv_0\in L^2(\Omega,\F_0,\p;L^2(G))$ with $\Lambda_0=\p\circ \bfv_0^{-1}$,
\item $\bff\in L^2(\Omega,\F,\p;L^2(\Q))$ is adapted to $(\mathcal F_t)$ and $\Lambda_\bff=\p\circ \bff^{-1}$,
\item for all
 $\bfvarphi\in C^\infty_{0,\Div}(G)$ and all $t\in[0,T]$ there holds $\p$-a.s.
\begin{align*}
\int_G\bfv(t)\cdot\bfvarphi\dx &+\int_0^t\int_G\bfS(\ep(\bfv)):\ep(\bfphi)\dxs-\int_0^t\int_G\bfv\otimes\bfv:\ep(\bfphi)\dxs\\&=\int_G\bfv_0\cdot\bfvarphi\dx
+\int_G\int_0^t\bff\cdot\bfphi\dxs+\int_G\int_0^t\Phi(\bfv)\,\dd\bfW_\sigma\cdot \bfvarphi\dx.
\end{align*}
\end{enumerate}
\end{definition}

\begin{remark}
\label{rem:mar}
\begin{itemize}
\item As we are looking for martingale solutions (weak solutions in the probabilistic sense) we can only assume the laws of $\bfv_0$ and $\bff$.
\end{itemize}
\end{remark}

\begin{theorem}[Existence]\label{thm:main}
Assume (\ref{1.4'}) with $p>\tfrac{2d+2}{d+2}$ as well as \eqref{eq:phi} and \eqref{eq:phi2}. Suppose further that
\begin{equation}\label{initial}
\int_{L^2_{\Div}(G)}\big\|\bfu\big\|_{L^2(G)}^\beta\,\dd\Lambda_0(\bfu)<\infty,\quad
\int_{L^2(\Q)}\big\|\bfg\big\|_{L^2(\Q)}^\beta\,\dd\Lambda_\bff(\bfg)<\infty,
\end{equation}
with $\beta:=\max\big\{\frac{2(d+2)}{d},\frac{p(d+2)}{d}\big\}$. Then there is a martingale weak solution to (\ref{eq:}) in the sense of Definition \ref{def:weak}.
\end{theorem}

\begin{remark}
\begin{itemize}
\item By Theorem \ref{thm:main} we extend the results from \cite{Wo} to the stochastic fashion where we can consider arbitrary bounded Lipschitz domains and allow a nonlinear dependence between the solution and the stochastic perturbation. The bound $p>\tfrac{8}{5}$ (if $d=3$) includes a wide range of Non-Newtonian fluids.
\item It is not clear if it is possible to improve the result from Theorem \ref{thm:main} to $p>\frac{2d}{d+2}$ as in the deterministic case. The papers \cite{DRW} and \cite{BrDS} use the Lipschitz truncation method. Despite the $L^\infty$-truncation the Lipschitz truncation is not only nonlinear but also nonlocal (in space-time in the parabolic case). So it seems to be impossible to perform the testing with it via It\^{o}'s formula.
\item Condition \eqref{initial} ensures the existence of higher moments for initial datum and right-hand-side. 
\end{itemize}
\end{remark}

\section{Pressure decomposition}
\label{sec:stochpi}
In this section we introduce the pressure and decompose it in way that every part of the pressure corresponds to one term in the equation. The following theorem generalizes \cite{Wo}[Thm. 2.6] to the stochastic case. 

\begin{theorem}\label{thm:pi}
$\left.\right.$\\
Let $(\Omega,\mathcal F,(\mathcal F_t),\p)$ be a stochastic basis, $\bfu\in L^2(\Omega,\F,\p;L^\infty(0,T;L^2(G)))$, $\bfH\in L^{s}(\Omega,\F,\p;L^{s}(\Q))$ for some $s>1$, both adapted to $(\mathcal F_t)$. Moreover, let $\bfu_0\in L^2(\Omega,\F_0,\p;L^2_{\Div}(G))$\\ and $\Phi \in L^2(\Omega,\F,\p;L^\infty(0,T;L_2(U,L^2(G))))$ progressively measurable such that
\begin{align*}
\int_G\bfu(t)\cdot\bfvarphi\dx +\int_0^t\int_G\bfH:\nabla\bfphi\dxs&=\int_G\bfu_0\cdot\bfvarphi\dx+\int_G\int_0^t\Phi\,\dd\bfW_\sigma\cdot \bfvarphi\dx
\end{align*}
holds for all $\bfphi\in C^\infty_{0,\Div}(G)$. 
Then there are functions $\pi_\bfH$, $\pi_h$ and $\pi_\Phi$ adapted to $(\mathcal F_t)$ such that 
\begin{enumerate}
\item[a)] We have $\Delta \pi_h=0$ and there holds for $\chi:=\min\left\{2,s\right\}$
\begin{align*}
 \E\bigg[\int_{\Q}|\pi_\bfH|^s\dxt\bigg]&\leq c\,\E\bigg[ \int_{\Q}|\bfH|^s\dxt\bigg],\\
\E\bigg[\sup_{(0,T)}\int_{G}|\pi_\Phi|^2\dx\bigg]
&\leq c\,\E\bigg[\sup_{(0,T)}\|\Phi\|_{L_2(U,L^2(G))}^2\bigg],\\
\E\bigg[\sup_{(0,T)}\int_{G}|\pi_h|^\chi\dx\bigg]
&\leq c\,\E\bigg[1+\sup_{(0,T)}\int_{G}|\bfu|^2\dx+\sup_{(0,T)}\|\Phi\|_{L_2(U,L^2(G))}^2 \bigg]\\ &+c\,\E\bigg[\int_{G}|\bfu_0|^2\dx+\int_{\Q}|\bfH|^s\dxt\bigg].
\end{align*}
\item[b)] There holds
\begin{align*}
\int_G&\big(\bfu(t)-\nabla\pi_h(t)\big)\cdot\bfvarphi\dx +\int_0^t\int_G\bfH:\nabla\bfphi\dxs-\int_0^t\int_G\pi_\bfH\,\Div\bfphi\dxs\\&=\int_G\bfu_0\cdot\bfvarphi\dx+\int_G\pi_\Phi(t)\,\Div\bfphi\dx+\int_G\int_0^t\Phi\,\dd\bfW_\sigma\cdot \bfvarphi\dx
\end{align*}
for all $\bfphi\in C^\infty_{0}(G)$. Moreover, we have $\pi_h(0)=\pi_\bfH(0)=\pi_\Phi(0)=0$ $\p$-a.s.
\end{enumerate}
\end{theorem}
\begin{remark}
If we put the pressure terms together by
\begin{align*}
\pi(t)=\pi_h(t)+\pi_\Phi(t)+\int_0^t\pi_\bfH\ds
\end{align*}
then there holds $\pi\in L^\chi(\Omega,\F,\p; L^\infty(0,T;L^\chi(G)))$.
\end{remark}
\begin{proof}
Let $\bfu$ be a weak solution to
\begin{align*}
\int_G\bfu(t)\cdot\bfvarphi\dx +\int_0^t\int_G\bfH:\nabla\bfphi\dxs&=\int_G\bfu_0\cdot\bfvarphi\dx+\int_0^t\int_G\Phi\,\dd\bfW_\sigma\cdot \bfvarphi\dx
\end{align*}
for all $\bfvarphi\in W^{1,\chi'}_{0,\Div}(G)$. Then there is a unique function $\pi(t)\in L^{\chi}_0(G)$ with $\pi(0)=0$ such that
\begin{align*}
\int_G\bfu(t)\cdot\bfvarphi\dx &+\int_0^t\int_G\bfH:\nabla\bfphi\dxs\\&=\int_G\pi(t)\,\Div\bfphi\dx+\int_G\bfu_0\cdot\bfvarphi\dx+\int_G\int_0^t\Phi\,\dd\bfW_\sigma\cdot \bfvarphi\dx
\end{align*}
for all $\bfvarphi\in W^{1,\chi'}_{0}(G)$. This is a consequence of the well-known Theorem by De Rahm. We will show
\begin{align}\label{eq:piinfty}
\pi\in L^\chi(\Omega,\F,\p; L^\infty(0,T;L^\chi(G))).
\end{align}
The measurability of $\pi$ follows from the equation. For the boundedness we write the equation as
\begin{align*}
\int_G\pi(t)\,\varphi\dx &=\int_G\big(\bfu(t)-\bfu_0\big)\cdot\mathcal B(\varphi)\dx-\int_0^t\int_G\bfH:\nabla \mathcal B(\varphi)\dxs\\
&+\int_G\int_0^t\Phi\,\dd\bfW_\sigma\cdot \mathcal B(\varphi)\dx,\quad \mathcal B(\varphi):=\Bog_G\big(\varphi-(\varphi)_G\big),
\end{align*}
for all $\varphi\in C^\infty_0(G)$ with the \Bogovskii -operator $\Bog_G$ (see \cite{Bog}). Here $(\varphi)_G$ denotes the mean value of the function $\varphi$ over $G$. This yields
\begin{align*}
\pi(t)&=\mathcal B^\ast\big(\bfu(t)-\bfu_0\big)-\int_0^t\big(\nabla\mathcal B\big)^\ast\bfH\ds
+\int_0^t\mathcal B^\ast\Phi\,\dd\bfW_\sigma,
\end{align*}
where $\mathcal B^\ast$ denotes the adjoint of $\mathcal B$ with respect to the $L^2(G)$ inner product.
Using continuity of $\mathcal B^{*}$ from $L^2(G)$ to $L^2(G)$ and $\big(\nabla\mathcal B\big)^{*}$ from $L^s(G)$ to $L^s(G)$ (which follows from the properties of $\Bog_G$) we have
we have
\begin{align}\label{eq:inftypi}
\E\bigg[\sup_{(0,T)}\int_{G}|\pi|^\chi\dx\bigg]
&\leq c\,\E\bigg[\sup_{(0,T)}\int_{G}|\bfu|^2\dx+\int_{G}|\bfu_0|^2\dx+\int_0^T\|\Phi\|_{L_2(U,L^2(G))}^2\dt\bigg]\nonumber\\
&+c\,\E\bigg[1+ \int_{\Q}|\bfH|^s\dxt\bigg],
\end{align}
and so \eqref{eq:piinfty} holds.
Note that the estimate of the stochastic integral is a consequence of the infinite-dimensional Burgholder-Davis-Gundi inequality (and the continuity of $\mathcal B^{*}$ on $L^2(G)$).\\
We decompose pointwise on $\Omega\times(0,T)$
\begin{align*}
\pi&=\pi_0+\pi_h,\\
\pi_0&:=\Delta \Delta^{-2}_G\Delta \pi,\quad \pi_h:=\pi-\pi_0.
\end{align*}
Here $\Delta^{-2}_G$ denotes the solution operator to the bi-Laplace equation w.r.t. zero boundary values for function and gradient. Since the operator $\Delta \Delta^{-2}_G\Delta$ is continuous from $L^\chi(G)$ to $L^\chi(G)$ (see \cite{Mul}) inequality (\ref{eq:inftypi}) holds true if $\pi$ is replaced by $\pi_0$ or $\pi_h$.
We gain for all $\varphi\in C^\infty_0(G)$
\begin{align}\label{eq:pi0}
\int_G\pi_0(t)\,\Delta\phi\dx=-\int_G\int_0^t\bfH:\nabla^2\phi\dxs+\int_G\int_0^t\Phi\,\dd\bfW_\sigma\cdot \nabla\varphi\dx.
\end{align}
Note that $\pi_0(t)\in \Delta W^{2,\chi}_0(G)$ is uniquely determined as the solution to the equation above.\\
There is a function $\pi_\bfH\in\Delta W_0^{2,q}(G)$ such that
\begin{align*}
\int_G\pi_\bfH(t)\,\Delta\phi\dx=-\int_G\bfH:\nabla^2\phi\dx
\end{align*}
for all $\phi\in C^\infty_0(G)$. The measurability of $\pi_\bfH$ follows from the measurability of the r.h.s. and we have on account of the solvability of the bi-Laplace equation (see \cite{Mul}) 
\begin{align*}
 \int_G|\pi_\bfH|^s\dx\leq c \int_G|\bfH|^s\dx\quad\p\otimes \LL^1-a.e.
\end{align*}
which implies
\begin{align*}
 \int_{\Omega\times \Q}|\pi_\bfH|^s\dxt\,\dd\p\leq c \int_{\Omega\times \Q}|\bfH|^s\dxt\,\dd\p.
\end{align*}
Moreover, we define $\pi_\Phi(t):=\pi_0(t)-\int_0^t\pi_\bfH\ds\in\Delta W_0^{2,\chi}(G)$. This is the unique solution to
\begin{align}\label{eq:pi2}
\int_G\pi_\Phi(t)\,\Delta\phi\dx&=\int_G\int_0^t\Phi\,\dd\bfW_\sigma\cdot \nabla\varphi\dx,\quad\varphi\in C^\infty_0(G).
\end{align}
Note that we can write (\ref{eq:pi2}) as
\begin{align*}
\int_G\pi_\Phi(t)\,\phi\dx&=\int_G\int_0^t\Phi\,\dd\bfW_\sigma\cdot\nabla\big(\Delta^{-2}\Delta\varphi\big)\dx
\end{align*}
for all $\varphi\in C^\infty_0(G)$ since $\pi_\Phi(t)\in \Delta W^{2,\chi}_0(G)$. 
Introducing the operator $\mathcal D:=\nabla\Delta^{-2}_G\Delta:L^2(G)\rightarrow W_0^{1,2}(G)$ we have
\begin{align}\label{eq:DD0}
\mathcal D,\mathcal D^{*}:L^2(G)\rightarrow L^2(G)
\end{align}
and hence $\p\times \mathcal L^{d+1}$-a.e.
$\pi_\Phi(t)=\int_0^t\mathcal D^\ast\Phi\,\dd\bfW_\sigma$.
This yields by the infinite-dimensional Burgholder-Davis-Gundi inequality and \eqref{eq:DD0}
\begin{align*}
\E\bigg[\sup_{(0,T)}\int_G|\pi_\Phi|^2\dx\bigg]&\leq\,c\,\E\bigg[\sup_{(0,T)}\|\mathcal D^\ast\Phi\|_{L_2(U,L^2(G))}^2 \bigg]\leq\,c\,\E\bigg[\sup_{(0,T)}\|\Phi\|_{L_2(U,L^2(G))}^2 \bigg].
\end{align*}
Finally, we see that $\tilde{\pi}_0(t):=\pi_\Phi(t)+\int_0^t\pi_\bfH\ds$ solves (\ref{eq:pi0}) and there holds $\tilde{\pi}_0(t)\in \Delta W^{2,\chi}_0(G)$ which implies
\begin{align*}
\pi_0(t)=\pi_\Phi(t)+\int_0^t\pi_\bfH\ds.
\end{align*}
Therefore we have the equation claimed in b).
\end{proof}

\begin{corollary}\label{cor:neu}
$\left.\right.$\\
Let the assumptions of Theorem \ref{thm:pi} be satisfied. There is $\Phi_\pi \in L^2(\Omega,\F,\p;L^\infty(0,T;L_2(U,L^2_{loc}(G))))$ progressively measurable
such that
\begin{align*}
\int_G\pi_\Phi(t)\,\Div\bfphi\dx&=\int_G\int_0^t\Phi_\pi\,\dd\bfW_\sigma\,\cdot\bfvarphi\dx,\quad\bfvarphi\in C^\infty_0(G).
\end{align*}
Let $G'\Subset G$, then $\Phi_\pi$ satisfies $\|\Phi_\pi\bfe_k\|_{L^2(G')}\leq \,c(G')\,\|\Phi\bfe_k\|_{L^2(G)}$ for all $k$, i.e. we have $\p\otimes\mathcal L^1$-a.e.
\begin{align*}
\|\Phi_\pi\|_{L_2(U,L^2(G'))}\leq \,c(G')\,\|\Phi\|_{L_2(U,L^2(G))}.
\end{align*}
If we assume that $\Phi$ satisfies (\ref{eq:phi}) then there holds
\begin{align*}
\|\Phi_\pi(\bfu_1)-\Phi_\pi(\bfu_2)\|_{L_2(U,L^2(G'))}\leq \,c(G')\,\|\bfu_1-\bfu_2\|_{L^2(G)}
\end{align*}
for all $\bfu_1,\bfu_2\in L^2(G)$.
\end{corollary}
\begin{proof}
From the proof of Theorem \ref{thm:pi} we gain for any $\bfvarphi\in C^\infty_0(G)$
\begin{align*}
\int_G\pi_\Phi(t)\,\Div\bfphi\dx&=\int_G\int_0^t\Phi\,\dd\bfW_\sigma\cdot\nabla\big(\Delta^{-2}\Delta\Div\bfvarphi\big)\dx\\
&=\sum_k\int_G\int_0^t\Phi\bfe_k\,\dd\beta_k\cdot\nabla\big(\Delta^{-2}\Delta\Div\bfvarphi\big)\dx\\
&=\sum_k\int_G\int_0^t\nabla\Delta\Delta^{-2}\Div\Phi\bfe_k\,\dd\beta_k\cdot\bfvarphi\dx\\
&=\int_G\int_0^t\nabla\Delta\Delta^{-2}\Div\Phi\,\dd\bfW_\sigma\cdot\bfvarphi\dx.
\end{align*}
This yields the first claim by setting $\Phi_\pi=\nabla\Delta\Delta^{-2}\Div\Phi$. The rest is a consequence of local regularity theory for the bi-Laplace equation.
\end{proof}

\begin{remark}
If the boundary of $G$ is smooth then the statement of Corollary \ref{cor:neu} holds globally (i.e. we can replace $G'$ by $G$). In this case the operator
$\nabla\Delta\Delta^{-2}\Div$ is continuous on $L^2(G)$ (see \cite{CS1}, section 2.2, or \cite{CS2}).
\end{remark}

\begin{corollary}\label{cor:pir}
$\left.\right.$\\
Let the assumptions of Theorem \ref{thm:pi} be satisfied. Then we have for all $\beta\in[1,\infty)$
\begin{align*}
\E\bigg[\sup_{(0,T)}\int_{G}|\pi_h|^\chi\dx\bigg]^\beta
\leq c\,\E&\bigg[\sup_{(0,T)}\int_{G}|\bfu|^2\dx+\sup_{(0,T)}\|\Phi\|_{L_2(U,L^2(G))}^2\bigg]^\beta\\&+c\,\E\bigg[1+\int_{G}|\bfu_0|^2\dx+\int_{\Q}|\bfH|^s\dxt\bigg]^\beta
\end{align*}
provided the r.h.s. is finite.
\end{corollary}

\begin{corollary}\label{cor:pi}
$\left.\right.$\\
Let the assumptions of Theorem \ref{thm:pi} be satisfied. Assume further that  we have the decomposition
\begin{align*}
\bfH=\bfH_1+\bfH_2,
\end{align*}
where $\bfH_1\in L^{s_1}(\Omega\times \Q,\p\otimes\mathcal L^{d+1})$ and $\bfH_2,\nabla\bfH_2\in L^{s_2}(\Omega\times \Q,\p\otimes\mathcal L^{d+1})$. Then we have 
$$\pi_\bfH=\pi_{1}+\pi_{2}$$
and it holds for all $\beta<\infty$ and all $G'\Subset G$
\begin{align*}
\E\bigg[ \int_{\Q}|\pi_{1}|^{s_1}\dxt\bigg]^\beta&\leq c\,\E\bigg[\int_{\Q}|\bfH_1|^{s_1}\dxt\bigg]^\beta,\\
\E\bigg[ \int_{\Q}|\pi_{2}|^{s_2}\dxt\bigg]^\beta&\leq c\,\E\bigg[\int_{\Q}|\bfH_2|^{s_2}\dxt\bigg]^\beta,\\\E\bigg[ \int_0^T\int_{G'}|\nabla\pi_{2}|^{s_2}\dxt\bigg]^\beta&\leq c\,\E\bigg[\int_{\Q}|\bfH_2|^{s_2}+|\nabla\bfH_2|^{s_2}\dxt\bigg]^\beta.
\end{align*}
\end{corollary}

\begin{proof}
$\pi_{1}$ and $\pi_{2}$ are the unique solutions (defined $\p\otimes\LL^1$-a.e.) to
\begin{align*}
\int_G\pi_{1}(t)\,\Delta\phi\dx=-\int_G\bfH_1:\nabla^2\phi\dx,\\
\int_G\pi_{2}(t)\,\Delta\phi\dx=-\int_G\bfH_1:\nabla^2\phi\dx,
\end{align*}
in the spaces $\Delta W_0^{2,s_1}(G)$ and $\Delta W_0^{2,s_2}(G)$. This gives immediately the claimed estimates (see \cite{Wo}, Lemma 2.3, for more details).
\end{proof}

\section{The approximated system}
\label{sec:stochappr}
We stabilize the equation by adding a large power of the velocity. 
For $\alpha>0$ we study the system
\begin{align}\label{eq:stokes}
\begin{cases}\dd\bfv=&\Div\bfS(\ep(\bfv))\dt-\alpha\,|\bfv|^{q-2}\bfv\dt+\nabla\pi\dt\\&-\Div\big(\bfv\otimes\bfv\big)\dt+\bff\dt+\Phi(\bfv)\,\dd\bfW_t\\
\bfv(0)=&\bfv_0\end{cases},
\end{align}
depending on
the law $\Lambda_\bff$ on $L^2(\Q)$ and $\Lambda_0$ on $L^2_{\Div}(G)$. In fact, we fix some $\bff\in L^2(\Omega,\F,\p; L^2(\Q))$ with $\Lambda_\bff=\p\circ \bff^{-1}$ and some $\bfv_0\in L^2(\Omega,\F_0,\p;L^2_{\Div}(G))$ with $\Lambda_0=\p\circ \bfv_0^{-1}$. By enlarging the filtration $(\mathcal F_t)$ we can assume that $\bff$ is adapted to it.
We choose $q\geq\max\{2p',3\}$ thus the solution is also an admissible test function. We expect a solution $\bfv$ in the space
\begin{align*}
\mathcal V_{p,q}:=&L^2(\Omega,\F,\p;L^\infty(0,T;L^2(G)))\cap L^q(\Omega\times \Q;\p\otimes\mathcal L^{d+1})\\&\cap L^p(\Omega,\F,\p;L^p(0,T;W^{1,p}_{0,\Div}(G))).
\end{align*}
We will try to find a solution by separating space and time via a Galerking-Ansatz which yields an approximated solution by solving an ordinary stochastic differential equation.

There is a sequence $(\lambda_k)\subset\R$ and a sequence of functions $(\bfw_k)\subset W_{0,\Div}^{l,2}(G)$, $l\in\N$, such that\footnote{see \cite{MNRR}, appendix}
\begin{itemize}
\item[i)] $\bfw_k$ is an eigenvector to the eigenvalue $\lambda_k$ of the Stokes-operator in the sense that:
$$\langle \bfw_k,\bfphi\rangle_{W_0^{l,2}}= \lambda_k\int_G\bfw_k\cdot\bfvarphi\,dx\quad\text{for all }\bfvarphi\in W_{0,\Div}^{l,2}(G),$$
\item[ii)] $\int_{G}\bfw_k\bfw_m\,dx=\delta_{km}$ for all $k,m\in\mathbb{N}$,
\item[iii)] $1\leq\lambda_1\leq \lambda_2\leq...$ and $\lambda_k\rightarrow\infty$,
\item[iv)] $\langle\tfrac{\bfw_k}{\sqrt{\lambda_k}},\tfrac{\bfw_m}{\sqrt{\lambda_m}}\rangle_{W_0^{l,2}}=\delta_{km}$ for all $k,m\in\mathbb{N}$,
\item[v)] $(\bfw_k)$ is a basis of $W_{0,\Div}^{l,2}(G)$.
\end{itemize}
We choose $l>1+\frac{d}{2}$ such that $W_0^{l,2}(G)\hookrightarrow W^{1,\infty}(G)$.
We are looking for an approximated solution $\bfv^N$ of the form
\begin{align*}
\bfv^N=\sum_{k=1}^N c_i^N\bfw_k=\bfC^N\cdot\bfw^N,\quad \bfw^N=(\bfw_1,...,\bfw_N),
\end{align*}
where $\bfC^N=(c_i^N):\Omega\times (0,T)\rightarrow \R^N$. Therefore, we would like to solve the system ($k=1,...,N$)
\begin{align}
\int_G &\dd\bfv^N\cdot\bfw_k\dx+\int_{G}\bfS(\ep(\bfv^N)):\ep(\bfw_k)\dx\dt+\alpha\int_G|\bfv^N|^{q-2}\bfv^N\cdot\bfw_k\dxt\nonumber
\\&=\int_G\bfv^N\otimes\bfv^N:\nabla\bfw_k\dx\dt+\int_G\bff\cdot\bfw_k\dx\dt+\int_G \Phi(\bfv^N)\,\dd\bfW^N_\sigma\cdot\bfw_k\dx,\nonumber
\\
&\qquad\qquad\bfv^N(0)=\mathcal P^N\bfv_0.\label{eq:gal}
\end{align}
Here $\mathcal P^N:L^{2}_{\Div}(G)\rightarrow \mathcal X_N:=\mathrm{span}\left\{\bfw_1,...,\bfw_N\right\}$ is the orthogonal projection, i.e.
\begin{align*}
\mathcal P^N\bfu=\sum_{k=1}^N\langle \bfu,\bfw_k\rangle_{L^{2}}\bfw_k.
\end{align*} 
The equation above is to be understood $\mathbb P$ a.s. and for a.e. $t$ and we assume
\begin{align*}
\bfW^N(\sigma)=\sum_{k=1}^N\bfe_k \beta_k(\sigma)=\bfe^N\cdot \bfbeta^N(\sigma).
\end{align*}
It is equivalent to solving
\begin{align}\label{SDE*}
\begin{cases}\dd\bfC^N&=\big[\bfmu (t,\bfC^N)\big]\dt+\bfSigma (\bfC^N)\,\dd\bfbeta^N_t\\
\bfC^N(0)&=\bfC_0\end{cases}
\end{align}
with the abbreviations
\begin{align*}
\bfmu (\bfC^N)&=\bigg(-\int_{G}\bfS(\bfC^N\cdot\ep( \bfw^N)):\ep(\bfw_k)\dx+\int_{G}(\bfC^N\cdot\bfw^N)\otimes (\bfC^N\cdot\bfw^N):\nabla\bfw_k\dx\bigg)_{k=1}^N\\
&-\bigg(\alpha\int_{G}|\bfC^N\cdot\bfw^N|^{q-2}(\bfC^N\cdot\bfw^N)\cdot\bfw_k\dx\bigg)_{k=1}^N
+\bigg(\int_{G}\bff(t)\cdot\bfw_k\dx\bigg)_{k=1}^N,\\
\bfSigma(\bfC^N)&=\bigg(\int_G \Phi(\bfC^N\cdot\bfw^N)\bfe_l\cdot \bfw_k\dx\bigg)_{k,l=1}^N,\\
\bfC_0&=\Big(\langle\bfv_0,\bfw_k\rangle_{L^{2}(G)}\Big)_{k=1}^N.
\end{align*}
If both $\bfmu$ and $\bfSigma$ are globally Lipschitz continuous one could quote the classical existence theorems for SDEs from \cite{Ar} and \cite{Fr1,Fr2}. This is not given in our situation, so
we apply more recent results from \cite{PrRo}, Thm. 3.1.1. In the following we will check the assumptions. we have by the monotonicity of $\bfS$
\begin{align*}
\big(\bfmu(t,\bfC^N)&-\bfmu(t,\tilde{\bfC}^N)\big)\cdot\big(\bfC^N-\tilde{\bfC}^N\big)\\
&=-\int_G\big(\bfS(\ep(\bfv^N))-\bfS(\ep(\tilde{\bfv}^N))\big):\big(\ep(\bfv^N)-\ep(\tilde{\bfv}^N)\big)\dx\\
&+\int_G\big(\bfv^N\otimes \bfv^N-\tilde{\bfv}^N\otimes \tilde{\bfv}^N\big):\big(\ep(\bfv^N)-\ep(\tilde{\bfv}^N)\big)\dx\\
&\leq \int_G\big(\bfv^N\otimes \bfv^N-\tilde{\bfv}^N\otimes \tilde{\bfv}^N\big):\big(\ep(\bfv^N)-\ep(\tilde{\bfv}^N)\big)\dx.
\end{align*}
If $|\bfC^N|\leq R$ and $|\tilde{\bfC}^N|\leq R$ there holds
\begin{align*}
\big(\bfmu(t,\bfC^N)&-\bfmu(t,\tilde{\bfC}^N)\big)\cdot\big(\bfC^N-\tilde{\bfC}^N\big)
\leq c(R,N)|\bfC^N-\tilde{\bfC}^N|^2.
\end{align*}
Here we took into account boundedness of $\bfw_k$ and $\nabla\bfw_k$.
This implies weak monotonicity in the sense of \cite{PrRo}, (3.1.3) using Lipschitz continuity $\bfSigma$ in $\bfC^N$, cf. (\ref{eq:phi}).
On account of $\int_G \bfv^N\otimes\bfv^N:\ep(\bfv^N)\dx=0$ there holds
\begin{align*}
\bfmu(t,\bfC^N)\cdot\bfC^N&=-\int_G\bfS(\ep(\bfv^N)):(\ep(\bfv^N)\dx+\int_G\bff(t)\cdot\bfv^N\dx\leq c\,(1+\|\bff(t)\|_2\|\bfv^N\|_2)\\
&\leq (1+\|\bff(t)\|_2)(1+\|\bfv^N\|^2)\leq \,c\,(1+\|\bff(t)\|_2)(1+|\bfC^N|^2)
\end{align*}
So we have using the linear growth of $\bfSigma$ which follows from \ref{eq:phi}
\begin{align*}
\bfmu(\bfC^N)\cdot\bfC^N+|\bfSigma(\bfC^N)|^2\leq c(+\|\bfv^N\|_2^2)\big(1+|\bfC^N|^2\big).
\end{align*}
As the integral $\int_0^T (1+\|\bff(t)\|_2)\dt$ is finite $\p$-a.s. this yields weak coercivity in the sense of \cite{PrRo}, (3.1.4). We obtain a unique strong solution $\bfC^N\in L^2(\Omega,\F,\p;L^\infty(0,T))$ to the SDE (\ref{SDE*}).\\\

We obtain the following a priori estimate.

\begin{theorem}\label{thm:2.1}
Assume (\ref{1.4'}) with $p\in(1,\infty)$, (\ref{eq:phi}), $q\geq\{2p',3\}$ and
\begin{equation}\label{initial'}
\int_{L^2_{\Div}(G)}\big\|\bfu\big\|_{L^2(G)}^2\,\dd\Lambda_0(\bfu)<\infty,\quad
\int_{L^2(\Q)}\big\|\bfg\big\|_{L^2(\Q)}^2\,\dd\Lambda_\bff(\bfg)<\infty.
\end{equation}
Then there holds uniformly in $N$:
\begin{align*}
\E&\bigg[\sup_{t\in(0,T)}\int_G |\bfv^N(t)|^2\dx+\int_{\Q} |\nabla \bfv^N|^p\dxt+\alpha\int_{\Q} |\bfv^N|^q\dxt\bigg]\\&\leq c\,\bigg(1+
\int_{L^2_{\Div}(G)}\big\|\bfu\big\|_{L^2(G)}^2\,\dd\Lambda_0(\bfu)+
\int_{L^2(\Q)}\big\|\bfg\big\|_{L^2(\Q)}^2\,\dd\Lambda_\bff(\bfg)\bigg),
\end{align*}
where $c$ is independent of $\alpha$.
\end{theorem}

\begin{proof}
We apply It\^{o}'s formula to the function $f(\bfC)=\tfrac{1}{2}|\bfC|^2$ which shows
\begin{align}
\frac{1}{2}\|\bfv^N(t)\|_{L^2(G)}^2
&=\frac{1}{2}\|\bfC^N(0)\|_{L^2(G)}^2+\sum_{k=1}^N\int_0^t  c_k^N\,\dd(c_k^N)_\sigma+\frac{1}{2}\sum_{k=1}^N\int_0^t \,\dd\langle c_k^N\rangle_\sigma\nonumber\\
&=\frac{1}{2}\|\mathcal P^N\bfv_0\|_{L^2(G)}^2-\int_0^t\int_{G}\bfS(\ep(\bfv^N)):\ep(\bfv^N)\dxs\nonumber\\
&-\alpha\int_0^t\int_{G}|\bfv^N|^q\dxs+\int_0^t\int_{G}\bff\cdot\bfv^N\dxs\label{eq:2.1neu}\\&+\int_G\int_0^t\bfv^N\cdot \Phi(\bfv^N)\,\dd\bfW^N_\sigma\dx
+\frac{1}{2}\int_G\int_0^t\,\dd\Big\langle\int_0^\cdot\Phi(\bfv^N)\,\dd\bfW^N\Big\rangle_\sigma\dx.\nonumber
\end{align}
Here we used $\dd\bfv^N=\sum_{k=1}^N \dd c_k^N\bfw_k$, $\int_G\bfv^N\otimes\bfv^N:\nabla\bfv^N\dx=0$, property (ii) of the base $(\bfw_k)$  and
\begin{align*}
\dd c_k^N&=-\int_{G}\bfS(\ep(\bfv^N)):\ep(\bfw_k)\dxt-\alpha\int_{G}|\bfv^N|^{q-2}\bfv^N\cdot\bfw_k\dxt\\&+\int_{G}\bfv^N\otimes\bfv^N:\nabla\bfw_k\dxt+\int_{G}\bff\cdot\bfw_k\dxt+\int_G \Phi(\bfv^N)\,\dd\bfW^N_t\cdot\bfw_k\dx.
\end{align*}
Now we can follow, building expectations and using (\ref{1.4'}) together with Korn's inequality, that
\begin{align*}
\E\bigg[\int_G &|\bfv^N(t)|^2\dx+\int_0^t\int_G |\nabla \bfv^N|^p\dxs+\alpha\int_0^t\int_G | \bfv^N|^q\dxs\bigg]\\&\leq c\Big(1+\E\big[\|\bfv_0\|^2_{L^2(G)}\big]+\E\big[J_1(t)\big]+\E\big[J_2(t)\big]+\E\big[J_3(t)\big]\Big).
\end{align*}
Here we abbreviated
\begin{align*}
J_1(t)&=\int_0^t\int_G\bff\cdot\bfv^N\dxs,\\
J_2(t)&=\int_G\int_0^t\bfv^N\cdot \Phi(\bfv^N)\,\dd\bfW^N_\sigma\dx,\\
J_3(t)&=\int_G\int_0^t\,\dd\Big\langle\int_0^\cdot\Phi(\bfv^N)\,\dd\bfW^N\Big\rangle_\sigma\dx.
\end{align*}
Straightforward calculations show on account of (\ref{eq:W}) and (\ref{eq:phi})
\begin{align*}
\E[J_3]&=\E\bigg[\sum_{i=1}^N\int_0^t\int_G  |\Phi(\bfv^N)\bfe_i|^2\dxs\bigg]\\
&\leq\E\bigg[\sum_{i=1}^\infty\int_0^t\int_G  |g_i(\bfv^N)|^2\dxs\bigg]\\
&\leq\E\bigg[1+\int_0^t\int_G |\bfv^N|^2\dxs\bigg].
\end{align*}
Using Young's inequality we gain for arbitrary $\delta>0$
\begin{align*}
\E[J_1]\leq\delta\E\bigg[\int_0^t\int_G |\bfv^N|^2\dxs\bigg] +c(\delta)\E\bigg[\int_0^t\int_G |\bff|^{2}\dxs\bigg].
\end{align*}
Clearly, we have $\E[J_2]=0$. So
interchanging the time-integral and the expectation value and applying Gronwall's Lemma leads to
\begin{align}\label{eq:4.4}
\begin{aligned}
\sup_{t\in(0,T)}\E&\bigg[\int_G |\bfv^N(t)|^2\dx\bigg]+\E\bigg[\int_\Q |\nabla \bfv^N|^p\dxt\bigg]\\&\leq c\,\E\bigg[1+\int_G |\bfv_0|^2\dx+\int_\Q |\bff|^{2}\dx\bigg].
\end{aligned}
\end{align}
We want to interchange supremum and expectation value. Similar arguments as before show by (\ref{eq:4.4})
\begin{align}\label{eq:4.5}
&\E\bigg[\sup_{t\in(0,T)}\int_G |\bfv^N(t)|^2\dx\bigg]\\&\leq c\,\E\bigg[1+\int_G |\bfv_0|^2\dx+\int_\Q |\bff|^2\dx+\int_0^T\int_G|\bfv^N|^2\dxt\bigg]+\E\bigg[\sup_{(0,T)}|J_2(t)|\bigg].\nonumber
\end{align}
On account of Burgholder-Davis-Gundi inequality, Young's inequality and (\ref{eq:phi}) we gain\footnote{Note that the paths of $\bfv^N$ in $L^2(G)$ are continuous in time $\p$-a.s.}
\begin{align*}
\E\bigg[\sup_{t\in(0,T)}|J_2(t)|\bigg]&=\E\bigg[\sup_{t\in(0,T)}\bigg|\int_0^t\int_G\bfv^N\cdot\Phi(\bfv^N)\dx\,\dd\bfW^N_\sigma\bigg|\bigg]\\
&=\E\bigg[\sup_{t\in(0,T)}\bigg|\int_0^t\sum_i\int_G\bfv^N\cdot\Phi(\bfv^N)\bfe_i\dx\,\dd\beta_i(\sigma)\bigg|\bigg]\\
&=\E\bigg[\sup_{t\in(0,T)}\bigg|\int_0^t\sum_i\int_G\bfv^N\cdot g_i(\bfv^N)\dx\,\dd\beta_i(\sigma)\bigg|\bigg]\\
&\leq c\,\E\bigg[\int_0^T\sum_i\bigg(\int_G\bfv^N\cdot g_i(\bfv^N)\dx\bigg)^2\dt\bigg]^{\frac{1}{2}}\\
&\leq c\,\E\bigg[\bigg(\int_0^T\bigg(\sum_{i=1}^N \int_G|\bfv^N|^2\dx\int_G |g_i(\bfv^N)|^2\dx\bigg)\dt\bigg]^{\frac{1}{2}}\\
&\leq c\,\E\bigg[1+\int_0^T\bigg( \int_G|\bfv^N|^2\dx\bigg) ^2\dt\bigg]^{\frac{1}{2}}\\
&\leq \delta\,\E\bigg[\sup_{t\in(0,T)}\int_G|\bfv^N|^2\dx\bigg]+c(\delta)\,\E\bigg[1+\int_0^T\int_G|\bfv^N|^2\dxt\bigg].
\end{align*}
This finally proves the claim for $\delta$ sufficiently small using (\ref{eq:4.4}) as well as $\Lambda_0=\p\circ\bfv_0^{-1}$ and $\Lambda_\bff=\p\circ\bff^{-1}$.
\end{proof}

\begin{theorem}\label{thm:2.2}
Assume (\ref{1.4'}) with $p\in(1,\infty)$, (\ref{eq:phi}), $q\geq\{2p',3\}$ and
\eqref{initial'}.
\begin{itemize}
\item[a)] There is a martingale weak solution $$\big((\underline{\Omega},\underline{\mathcal F},(\underline{\mathcal F}_t),\underline{\p}),\underline{\bfv},\underline{\bfv}_0,\underline{\bff},\underline{\bfW})$$
 to (\ref{eq:stokes}) in the sense that
\begin{itemize}
\item[i)] $(\underline{\Omega},\underline{\mathcal F},(\underline{\mathcal F}_t),\underline{\p})$ is a stochastic basis with a complete right-continuous filtration,
\item[ii)] $\underline{\bfW}$ is an $(\underline{\mathcal F}_t)$-cylindrical Wiener process,
\item[iii)] $\underline{\bfv}\in\underline{\mathcal V}_{p,q}$ is progressively measurable, where
\begin{align*}
\underline{\mathcal V}_{p,q}:=&L^2(\underline{\Omega},\underline{\F},\underline{\p};L^\infty(0,T;L^2(G)))\cap L^q(\underline{\Omega}\times \Q;\underline{\p}\otimes\mathcal L^{d+1})\\&\cap L^p(\underline{\Omega},\underline{\F},\underline{\p};L^p(0,T;W^{1,p}_{0,\Div}(G))).
\end{align*}
\item[iv)] $\underline{\bfv}_0\in L^2(\Omega,\F_0,\p;L^2(G))$ with $\Lambda_0=\underline\p\circ \underline\bfv_0^{-1}$,
\item[v)] $\underline{\bff}\in L^2(\underline{\Omega},\underline{\F},\underline{\p};L^2(\Q))$ is adapted to $(\underline{\F}_t)$ with $\Lambda_\bff=\underline{\p}\circ \underline{\bff}^{-1}$,
\item[vi)] for all
 $\bfvarphi\in C^\infty_{0,\Div}(G)$ and all $t\in[0,T]$ there holds $\underline{\p}$-a.s.
\begin{align*}
\int_G&\underline{\bfv}(t)\cdot\bfvarphi\dx +\int_0^t\int_G\bfS(\ep(\underline{\bfv})):\ep(\bfphi)\dxs+\alpha\int_0^t\int_G|\underline{\bfv}|^{q-2}\underline{\bfv}\cdot\bfphi\dxs\\&=\int_0^t\int_G\underline{\bfv}\otimes\underline{\bfv}:\ep(\bfphi)\dxs+\int_G\underline{\bfv}_0\cdot\bfvarphi\dx
+\int_G\int_0^t\underline{\bff}\cdot\bfphi\dxs+\int_G\int_0^t\Phi(\underline{\bfv})\,\dd\underline{\bfW}_\sigma\cdot \bfvarphi\dx.
\end{align*}
\end{itemize}
\item[b)] There holds
\begin{align*}
\underline{\E}&\bigg[\sup_{t\in(0,T)}\int_G |\underline{\bfv}(t)|^2\dx+\int_\Q |\nabla \underline{\bfv}|^p\dxt+\alpha\int_\Q |\underline{\bfv}^N|^q\dxt\bigg]\\&\leq c\,\bigg(1+\int_{L^2_{\Div}(G)}\big\|\bfu\big\|_{L^2(G)}^2\,\dd\Lambda_0(\bfu)+
\int_{L^2(\Q)}\big\|\bfg\big\|_{L^2(\Q)}^2\,\dd\Lambda_\bff(\bfg)\bigg).
\end{align*}
where $c$ is independent of $\alpha$.
\end{itemize}
\end{theorem}

\begin{proof}
From the a priori estimate in Theorem \ref{thm:2.1} we can follow the existence of functions $\bfv\in \mathcal V_{p,q}$ and functions $\tilde{\bfs}$ and $\tilde{\bfS}$ such that (after passing to a not relabeled subsequence)
\begin{align}\label{eq:conv}
\begin{aligned}
\bfv^N&\rightharpoondown \bfv \quad\text{in}\quad  L^p(\Omega,\F,\p;L^p(0,T;W_0^{1,p}(G))),\\
\bfv^N&\rightharpoondown \bfv \quad\text{in}\quad L^q(\Omega,\F,\p;L^q(\Q)),\\
\bfs(\bfv^N)&\rightharpoondown \tilde{\bfs} \quad\text{in}\quad L^{q'}(\Omega,\F,\p;L^{q'}(\Q)),\\
\bfS(\ep(\bfv^N))&\rightharpoondown \tilde{\bfS} \quad\text{in}\quad L^{p'}(\Omega,\F,\p;L^{p'}(\Q)),\\
\bfS(\ep(\bfv^N))&\rightharpoondown \tilde{\bfS} \quad\text{in}\quad L^{p'}(\Omega,\F,\p;L^{p'}(0,T;W_0^{-1,p'}(G))).
\end{aligned}
\end{align}
Moreover, there are $\tilde{\bfV}$ and $\tilde{\Phi}$ (recall (\ref{eq:phi}) and Theorem \ref{thm:2.1}) such that
\begin{align}
\label{eq:conv2'}
\begin{aligned}
\bfv^N\otimes\bfv^N&\rightharpoondown \tilde{\bfV}\quad\text{in}\quad L^{\frac{q}{2}}(\Omega,\F,\p;L^{\frac{q}{2}}(\Q)),\\
\Phi(\bfv^N)&\rightharpoondown \tilde{\Phi}\quad\text{in}\quad L^2(\Omega,\F,\p; L^2(0,T;L_2(U,L^2(G)))).
\end{aligned}
\end{align}
We want to establish
\begin{align}\label{eq:conv3}
\tilde{\bfV}=\bfv\otimes\bfv,\quad\tilde{\Phi}=\Phi(\bfv).
\end{align}
This will be a consequence of some compactness arguments. \\
We will follow ideas from \cite{Ho1}, section 4. We consider $\bfphi\in C^\infty_{0,\Div}(G)$ and gain by (\ref{eq:gal})
\begin{align*}
\int_G\bfv^N&(t)\cdot\bfphi\dx=\int_G\bfv^N(t)\cdot\mathcal P^N_l\bfphi\dx\\
&=\int_G\bfv_0\cdot\mathcal P^N_l\bfphi\dx+\int_0^t\int_G \bfH^N:\nabla\mathcal P^N_l\bfphi\dxs+\int_G\int_0^t\Phi(\bfv^N)\dd\bfW^N_\sigma\cdot\mathcal P^N_l\bfphi\dx,\\
\bfH^N&:=-\bfS(\ep(\bfv^N))+\nabla\Delta^{-1}\bfs(\bfv^N)+\bfv^N\otimes\bfv^N-\nabla\Delta^{-2}\bff.
\end{align*}
Here $\mathcal P^N_l$ denotes the projection into $\mathcal X_N$ with respect to the $W^{l,2}_{0,\Div}(G)$ inner product. From the a priori estimates in Theorem \ref{thm:2.1} and the growth conditions for $\bfS$ following from (\ref{1.4'}) and $\bfs$ we gain
\begin{align}
\label{eq:p0}\bfH^N\in L^{p_0}(\Omega\times\Q;\p\otimes\mathcal L^{d+1}),\quad p_0:=\min\left\{p',q',\frac{q}{2}\right\}>1,
\end{align}
uniformly in $N$.
Let us consider the functional
$$\mathscr H(t,\bfphi):=\int_0^t\int_G \bfH^N:\nabla\mathcal P^N_l\bfphi\dxs,\quad \bfphi\in C^\infty_{0,\Div}(G).$$
Then we deduce from (\ref{eq:p0}) and the embedding $W^{\tilde{l},p_0}(G)\hookrightarrow W_0^{l,2}(G)$ for $\tilde{l}\geq l+d\big(1+\frac{2}{p_0}\big)$ the estimate
\begin{align*}
\E\bigg[\big\|\mathscr H\big\|_{W^{1,p_0}([0,T];W_{\Div}^{-\tilde{l},p_0}(G))}\bigg]\leq c.
\end{align*}
For the stochastic term we quote from \cite{Ho1}, proof of Lemma 4.6, for some $\mu=\mu(d,p)>0$
\begin{align*}
\E\bigg[\Big\|\int_0^t\Phi(\bfv^N)\,\dd\bfW_\sigma^N\Big\|_{C^{\mu}([0,T]; L^2(G))}\bigg]\leq c\bigg(1+\int_{\Omega\times \Q}|\bfv^N|^{q}\dxt\,\dd\p\bigg)\leq c.
\end{align*}
This is a consequence of Theorem \ref{thm:2.1}, $q>2$ and assumption (\ref{eq:phi}). Combining the both informations above shows
\begin{align}\label{eq:vt1}
\E\Big[\|\bfv^N\|_{C^{\mu}([0,T]; W_{\Div}^{-\tilde{l},p_0}(G))}\Big]\leq c.
\end{align}
and also for some $\lambda>0$
\begin{align}\label{eq:vt2}
\E\Big[\|\bfv^N\|_{W^{\lambda,p_0}(0,T; W_{0,\Div}^{-\tilde{l},p_0}(G))}\Big]\leq c.
\end{align}
An interpolation with $L^{p_0}(0,T; W_{0,\Div}^{1,p_0}(G))$ yields on account of (\ref{eq:aprimpro}) for some $\kappa>0$ (see \cite{Am}, Thm. 5.2)
\begin{align}\label{eq:vt3}
\E\Big[\|\bfv^N\|_{W^{\kappa,p_0}(0,T; L^{p_0}_{\Div}(G))}\Big]\leq c.
\end{align}
So, we have
\begin{align*}
W^{\kappa,p_0}(0,T;L^{p_0}_{\Div}(G))\cap \mathcal V_{p,q}\hookrightarrow \hookrightarrow L^r(0,T;L^r_{\Div}(G))
\end{align*}
compactly for all $r<q$. We will use this embedding in order to show compactness of $\bfv^N$.
We consider the path space
$$ \mathscr V:= L^r(0,T;L^r(G))\otimes C([0,T],U_0)\otimes L^2_{\Div}(G)\otimes L^2(\Q)$$ 
Here we will use the following notations:
\begin{itemize}
\item $\nu_{\bfv^N}$ is the law of $\bfv^N$ on $L^r(0,T;L^r(G))$;
\item $\nu_{\bfW}$ is the law of $\bfW$ on $C([0,T],U_0)$, where $U_0$ is defined in (\ref{eq:U0});
\item $\bfnu^m$ is the joint law of $\bfv^N$, $\bfW$, $\bfv_0$ and $\bff$ on $ \mathscr V$.
\end{itemize}
We consider the ball $\mathcal B_R$ in the space $W^{\kappa,p_0}(0,T;L^{p_0}_{\Div}(G))\cap \mathcal V_{p,q}$
and gain for its complement $\mathcal B_R^C$ by Theorem \ref{thm:2.1} and
\eqref{eq:vt3}
\begin{align*}
\mu_{\bfv^N}&(\mathcal B_R^C)=\p\big(\|\bfv^N\|_{W^{\kappa,p_0}(0,T;L^{p_0}_{\Div}(G))}+\|\bfv^N\|_{ \mathcal V_{p,q}}\geq R\big)\\&
\leq \frac{1}{R}\,\E\Big[\|\bfv^N\|_{W^{\kappa,p_0}(0,T;L^{p_0}_{\Div}(G))}+\|\bfv^N\|_{ \mathcal V_{p,q}}\Big]
\leq \frac{c}{R}.
\end{align*}
So for a fixed $\eta>0$ we find $R(\eta)$ with
\begin{align*}
\mu_{\bfu^m}(\mathcal B_{R(\eta)})&\geq 1-\frac{\eta}{4}.
\end{align*}
Since also the law $\mu_{\bfW}$ is tight as
being a Radon measure on the Polish space $C([0,T],U_0)$, there exists a compact set
$C_\eta\subset C([0,T],U_0)$ such that $\mu_{\bfW}(C_\eta)\geq1-\tfrac{\eta}{4}$. For the same reason we find compact subsets of $L^2_{\Div}(G)$ and $L^2(\Q)$ with such that their measures ($\Lambda_0$ and $\Lambda_\bff$) are smaller than $1-\frac{\eta}{4}$.
Hence, we can find a compact subset 
$\mathscr V_\eta\subset \mathscr V$
such that $\bfnu^m(\mathscr V_\eta)\geq1-\eta$. Thus, $\left\{\bfnu^N,\,\,N\in\N\right\}$ is tight in the same space.
Prokhorov's
Theorem (see \cite{IkWa}, Thm. 2.6, p. 7) therefore implies that $\bfnu^N$ is also relatively weakly compact. This means we have a weakly convergent subsequence with limit $\bfnu$. Now we use Skorohod's
representation theorem (see \cite{IkWa}, Thm. 2.7, p. 9) to infer the existence of a probability space $(\underline{\Omega},\underline{\F},\underline{\p})$, a sequence $(\underline{\bfv}^N,\underline{\bfW}^N,\underline{\bfv}_0^N,\underline{\bff}^N)$ and $(\underline{\bfv},\underline{\bfW},\underline{\bfv}_0,\underline{\bff})$ on $(\underline{\Omega},\underline{\F},\underline{\p})$ both with values in $\mathscr V$ such that the following holds
\begin{itemize}
\item The laws of $(\underline{\bfv}^N,\underline{\bfW}^N,\underline{\bfv}_0^N,\underline{\bff}^N)$ and $(\underline{\bfv},\underline{\bfW},\underline{\bfv}_0,\underline{\bff})$ under $\underline{\p}$ coincide with $\bfnu^N$ and $\bfnu$.
\item We have the convergences
\begin{align*}
\underline{\bfv}^N&\longrightarrow \underline{\bfv}\quad\text{in}\quad L^r(0,T;L^r(G)),\\
\underline{\bfW}^N&\longrightarrow \underline{\bfW}\quad\text{in}\quad C([0,T],U_0),\\
\underline{\bfv}_0^N&\longrightarrow \underline{\bfv}_0\quad\text{in}\quad L^2(G),\\
\underline{\bff}^N&\longrightarrow \underline{\bff}\quad\text{in}\quad L^2(0,T;L^2(G)),
\end{align*}
$\underline{\p}$-a.s.
\item The convergences in (\ref{eq:conv}) and (\ref{eq:conv2'}) remain valid for the corresponding functions defined on $(\underline{\Omega},\underline{\F},\underline{\p})$. Moreover, we have for all $\alpha<\infty$
\begin{align*}
\int_{\underline{\Omega}}\bigg(\sup_{[0,T]}\|\underline{\bfW}^N(t)\|_{U_0}^\alpha\bigg)\,\dd\underline{\p}=\int_{\Omega}\bigg(\sup_{[0,T]}\|\bfW(t)\|_{U_0}^\alpha\bigg)\,\dd\p.
\end{align*}
\end{itemize}
After choosing a subsequence we gain by Vitali's convergence Theorem
\begin{align}\label{eq:convWm'}
\underline{\bfW}^N&\longrightarrow\underline{\bfW}\quad\text{in}\quad L^2(\underline{\Omega},\underline{\F},\underline{\p};C([0,T],U_0)),\\\label{eq:compactdelta'}
\underline{\bfv}^N&\longrightarrow \underline{\bfv}\quad\text{in}\quad L^r(\underline{\Omega}\times \Q;\underline{\p}\otimes\mathcal L^{d+1}),\\
\label{eq:compactv0'}
\underline{\bfv}^N_0&\longrightarrow \underline{\bfv}_0\quad\text{in}\quad L^2(\underline{\Omega}\times G,\underline{\p}\otimes\mathcal L^{d+1}),\\
\label{eq:compactf'}
\underline{\bff}^N&\longrightarrow \underline{\bff}\quad\text{in}\quad L^2(\underline{\Omega}\times \Q,\underline{\p}\otimes\mathcal L^{d+1}),
\end{align}
for all $r<q$. Now we introduce the filtration on the new probability space.
We denote by $\bfr_t$ the operator of restriction to the interval $[0,t]$ acting on various path spaces. In particular, if $X$ stands for one of the path spaces $L^r(0,T;L^r(G)),\,L^2(Q)$ or $C([0,T],U_0)$ and $t\in[0,T]$, we define
\begin{align}\label{restr}
\bfr_t:X\rightarrow X|_{[0,t]},\quad f\mapsto f|_{[0,t]}.
\end{align}
Clearly, $ \bfr_t$ is a continuous mapping.
Let $(\underline{\mathcal F}_t)$ be the $\underline{\p}$-augmented canonical filtration of the process $\big(\underline{\bfv},\underline{\bff},\underline{\bfW}\big)$, respectively, that is
\begin{equation*}
\begin{split}
\underline{\mathcal F}_t&=\sigma\Big(\sigma\big(\bfr_t\underline{\bfv},\bfr_t\underline{\bff},\bfr_t\underline{\bfW}\big)\cup\big\{N\in\underline{\mathcal F};\;\underline{\p}(N)=0\big\}\Big),\quad t\in[0,T].
\end{split}
\end{equation*}
Now are going to show that the approximated equations also hold on the new probability
space. We use a general and elementary method that was recently introduced in \cite{on2} and already generalized to different settings (see for instance \cite{Ho1}). The keystone is to identify not only the quadratic variation of the corresponding martingale but also its cross variation with the limit Wiener process obtained through compactness. First we notice that $\underline \bfW^N$ has the same law as $\bfW$. As a consequence, there exists a collection of mutually independent real-valued $(\underline{\F}_t)$-Wiener processes $(\underline{\beta}^{N}_k)_{k}$ such that $\underline{\bfW}^N=\sum_{k}\underline{\beta}^{N}_k e_k$, i.e. there exists a collection of mutually independent real-valued $(\underline{\F}_t)$-Wiener processes $(\underline{\beta}_k)_{k\geq1}$ such that $\underline{\bfW}=\sum_{k}\underline{\beta}_k e_k$. We abbreviate $\underline{\bfW}^{N,N}:=\sum_{k=1}^N\bfe_k\underline{\beta}^N_k$.
Let us now define for all $t\in[0,T]$ and $\bfvarphi\in C^\infty_{0,\Div}(G)$ the functionals
\begin{equation*}
\begin{split}
M(\bfv^N,\bfv_0,\bff)_t&=\int_G\bfv^N(t)\cdot\bfvarphi\dx-\int_G \bfv_0\cdot\bfvarphi\dx+\int_0^t\int_G\bfv^N\otimes\bfv^N:\nabla\mathcal P_N\bfvarphi\dxs\\&+\int_0^t\int_G\bfS(\ep(\bfv^N)):\ep(\mathcal P_N\bfvarphi)\dxs
+\int_0^t\int_G\bff\cdot\mathcal P_N\bfvarphi\dxs,\\
N(\bfv^N)_t&=\sum_{k=1}^N\int_0^t\bigg(\int_G g_k(\bfv)\cdot\mathcal P_N\bfvarphi\dx\bigg)^2\ds,\\
N_k(\bfv^N)_t&=\int_0^t\int_G g_k(\bfv)\cdot\mathcal P_N\bfvarphi\dxs,
\end{split}
\end{equation*}
let $M(\bfv^N,\bfv_0,\bff)_{s,t}$ denote the increment $M(\bfv^N,\bfv_0,\bff)_{t}-M(\bfv^N,\bfv_0,\bff)_{s}$ and similarly for $N(\bfv^N)_{s,t}$ and $N_k(\bfv^N)_{s,t}$. 
Note that the proof will be complete once we show that the process $M(\underline\bfv^N)$ is an $(\underline{\F}_t)$-martingale and its quadratic and cross variations satisfy, respectively,
\begin{equation}\label{mart}
\begin{split}
\langle M(\underline\bfv^N,\underline{\bfv}_0,\underline\bff)\rangle&=N(\underline\bfv^N),\quad\qquad\langle M(\underline\bfv^N,\underline{\bfv}_0,\underline\bff),\underline \beta_k\rangle=N_k(\underline\bfv^N).
\end{split}
\end{equation}
Indeed, in that case we have
\begin{align}\label{neu3108}
\bigg\langle M(\underline\bfv^N,\underline{\bfv}_0,\underline\bff)-\int_0^{\cdot} \int_G\Phi(\underline\bfv^N)\,\dd\underline \bfW^{N,N}\cdot\mathcal P_N\bfvarphi\dx\bigg\rangle=0
\end{align}
which yields the desired equation on the new probability space.
Let us verify \eqref{mart}. To this end, we claim that with the above uniform estimates in hand, the mappings
$$(\bfv^N,\bfv_0,\bff)\mapsto M(\bfv^N,\bfv_0,\bff)_t,\quad\,\bfv^N\mapsto N(\bfv^N)_t,\,\quad \bfv^N\mapsto N_k(\bfv^N)_t$$
are well-defined and measurable on a subspace of the path space where the joint law of $(\underline\bfv^N,\underline{\bfv}_0,\underline\bff)$ is supported, i.e. the uniform estimates
form Theorem \ref{thm:2.1} hold true.
Indeed, in the case of $N(\rho,\bfq)_t$ we have by \eqref{eq:phi} and the continuity of $\mathcal P^N$ in $L^2(G)$
\begin{align*}
\sum_{k= 1}^N\int_0^t\bigg(\int_G g_k(\bfv^N)\cdot\mathcal P_N\varphi\dx\bigg)^2\ds&\leq c(\bfphi)\sum_{k=1}^\infty\int_0^t\int_G|g_k(\bfv^N)|^2\dxs\\
&\leq c(\bfphi)\bigg(1+\int_\Q|\bfv^N|^2\dxt\bigg)
\end{align*}
which is finite.
$M(\rho,\bfv,\bfq)$ and $N_k(\rho,\bfv)_t\,$ can be handled similarly and therefore, the following random variables have the same laws
\begin{align*}
M(\bfv^N,\bfv_0,\bff)&\sim M(\underline\bfv^N,\underline{\bfv}_0,\underline\bff),\\
N(\bfv^N)& \sim N(\underline\bfv^N),\\
N_k(\bfv^N)&\sim N_k(\underline\bfv^N).
\end{align*}
Let us Now fix times $s,t\in[0,T]$ such that $s<t$ and let
$$h:\mathscr V\big|_{[0,s]}\rightarrow [0,1]$$
be a continuous function.
Since
$$M(\bfv^N,\bfv_0,\bff)_t=\int_0^t\int_G\Phi(\bfv^N)\,\dd\bfW^N_\sigma\cdot\mathcal P_N\bfvarphi\dx=\sum_{k=1}^N\int_0^t\int_G g_k(\bfv^N)\cdot\mathcal P_N\bfvarphi\dx\,\dd\beta_k$$
is a square integrable $(\F_t)$-martingale, we infer that
$$\big[M(\bfv^N,\bfv_0,\bff)\big]^2-N(\bfv^N),\quad M(\bfv^N)\beta_k-N_k(\bfv^N),$$
are $(\F_t)$-martingales.
Let $\bfr_s$ be the restriction of a function to the interval $[0,s]$. Then it follows from the equality of laws that
\begin{equation*}
\begin{split}
&\underline{\E}\big[\,h\big(\bfr_s\underline\bfv^N,\bfr_s\underline{\bfW}^N,\bfr_s\underline{\bff},\underline\bfv_0\big)M(\underline\bfv^N,\underline{\bfv}_0,\underline\bff)_{s,t}\big]\\
&=\E \big[\,h\big(\bfr_s\bfv^N,\bfr_s\bfW^N,\bfr_s\bff,\bfv_0\big)M(\bfv^N,\bfv_0,\bff)_{s,t}\big]=0,
\end{split}
\end{equation*}
\begin{equation*}
\begin{split}
&\underline{\E}\bigg[\,h\big(\bfr_s\underline\bfv^N,\bfr_s\underline{\bfW}^N,\bfr_s\underline{\bff},\underline\bfv_0\big)\Big([M(\underline\bfv^N,\underline{\bfv}_0,\underline\bff)^2]_{s,t}-N(\underline\bfv^N)_{s,t}\Big)\bigg]\\
&=\E\bigg[\,h\big(\bfr_s\bfv^N,\bfr_s\bfW^N,\bfr_s\bff,\bfv_0\big)\Big([M(\bfv^N,\bfv_0,\bff)^2]_{s,t}-N(\bfv^N)_{s,t}\Big)\bigg]=0,
\end{split}
\end{equation*}
\begin{equation*}
\begin{split}
&\underline{\E}\bigg[\,h\big(\bfr_s\underline\bfv^N,\bfr_s\underline{\bfW}^N,\bfr_s\underline{\bff},\underline\bfv_0\big)\Big([M(\underline\bfv^N,\underline{\bfv}_0,\underline\bff)\underline{\beta}_k^N]_{s,t}-N_k(\underline\bfv^N)_{s,t}\Big)\bigg]\\
&=\E\bigg[\,h\big(\bfr_s\bfv^N,\bfr_s\bfW^N,\bfr_s\bff,\bfv_0\big)\Big([M(\bfv^N,\bfv_0,\bff)\beta_k]_{s,t}-N_k(\bfv^N)_{s,t}\Big)\bigg]=0.
\end{split}
\end{equation*}
So we have shown \eqref{mart} and hence (\ref{neu3108}). This means
on the new probability space $(\underline{\Omega},\underline{\F},\underline{\p})$ we have the equations (k=1,...,N)
\begin{align}
\int_G &\dd\underline{\bfv}^N\cdot\bfw_k\dx+\int_{G}\bfS(\ep(\underline{\bfv}^N)):\ep(\bfw_k)\dx\dt+\alpha\int_G|\underline{\bfv}^N|^{q-2}\underline{\bfv}^N\cdot\bfw_k\dxt\nonumber
\\&=\int_G\underline{\bfv}^N\otimes\underline{\bfv}^N:\nabla\bfw_k\dx\dt+\int_G\underline{\bff}\cdot\bfw_k\dx\dt+\int_G \Phi(\underline{\bfv}^N)\,\dd\underline{\bfW}^{N,N}_\sigma\cdot\bfw_k\dx,\nonumber
\\
&\qquad\qquad\underline{\bfv}^N(0)=\mathcal P^N\underline{\bfv}_0.\label{eq:gal'}
\end{align}
and the convergences
\begin{align}\label{eq:conv'}
\begin{aligned}
\underline{\bfv}^N&\rightharpoondown \underline{\bfv} \quad\text{in}\quad  L^p(\underline{\Omega},\underline{\F},\underline{\p};L^p(0,T;W_0^{1,p}(G))),\\
\underline{\bfv}^N&\rightharpoondown \underline{\bfv} \quad\text{in}\quad L^q(\underline{\Omega},\underline{\F},\underline{\p};L^q(\Q)),\\
\bfs(\underline{\bfv}^N)&\rightharpoondown \bfs(\underline{\bfv}) \quad\text{in}\quad L^{q'}(\underline{\Omega},\underline{\F},\underline{\p};L^{q'}(\Q)),\\
\bfS(\ep(\underline{\bfv}^N))&\rightharpoondown \underline{\tilde{\bfS}} \quad\text{in}\quad L^{p'}(\underline{\Omega},\underline{\F},\underline{\p};L^{p'}(\Q)),\\
\bfS(\ep(\underline{\bfv}^N))&\rightharpoondown \underline{\tilde{\bfS}} \quad\text{in}\quad L^{p'}(\underline{\Omega},\underline{\F},\underline{\p};L^{p'}(0,T;W_0^{-1,p'}(G))),\\
\underline{\bfv}^N\otimes\underline{\bfv}^N&\rightharpoondown \underline{\bfv}\otimes\underline{\bfv}\quad\text{in}\quad L^{\frac{q}{2}}(\underline{\Omega},\underline{\F},\underline{\p};L^{\frac{q}{2}}(\Q)),\\
\Phi(\underline{\bfv}^N)&\rightharpoondown \Phi(\underline{\bfv})\quad\text{in}\quad L^2(\underline{\Omega},\underline{\F},\underline{\p}; L^2(0,T;L_2(U,L^2(G)))).
\end{aligned}
\end{align}
We gain from (\ref{eq:convWm'})--(\ref{eq:conv'}) the limit equation
\begin{align}\label{eq:limitu'}
\int_G&\underline{\bfv}(t)\cdot\bfvarphi\dx+\int_0^t\int_G\underline{\tilde{\bfS}}:\nabla\bfphi\dxs+\int_0^t\int_G\bfs(\underline{\bfv}) \cdot\bfphi\dxs\\&
+\int_0^t\int_G\underline{\bfv}\otimes\underline{\bfv}:\nabla\bfphi\dxs=\int_0^t\int_G\underline{\bff}\cdot\bfphi\dxs+\int_G\int_0^t \Phi(\underline{\bfv})\,\dd\underline{\bfW}_\sigma\cdot \bfvarphi\dx\nonumber
\end{align}
 for all $\bfphi\in C^\infty_{0,\Div}(G)$. The limit in the stochastic term needs some explanations.
We have the convergences
\begin{align*}
\underline{\bfW}^N&\longrightarrow \underline{\bfW}\quad\text{in}\quad C([0,T],U_0),\\
\Phi(\underline{\bfv}^N)&\longrightarrow \Phi(\underline{\bfv})\quad\text{in}\quad L^2(0,T; L_2(U,L^2(G))),
\end{align*}
in probability. For the second one we use (\ref{eq:phi}) and (\ref{eq:compactdelta'}). These convergences imply
\begin{align*}
\int_0^t \Phi(\underline{\bfv}^N)\,\dd\underline{\bfW}^N_\sigma\longrightarrow\int_0^t \Phi(\underline{\bfv})\,\dd\underline{\bfW}_\sigma\quad\text{in}\quad L^2(0,T; L^2(G))
\end{align*}
in probability by \cite{DeGlTe}, Lemma 2.1. So we can pass to the limit in the stochastic integral.\\
Now, it remains to show
\begin{align}
\underline{\tilde{\bfS}}=\bfS(\ep(\underline{\bfv}))\label{eq:StildeS'}.
\end{align}
We will apply monotone operator theory to verify (\ref{eq:StildeS'}). Equation (\ref{eq:limitu'}) implies using $\int_G\underline{\bfv}\otimes\underline{\bfv}:\nabla\underline{\bfv}\dx=0$
\begin{align*}
\frac{1}{2}&\|\underline{\bfv}(t)\|_{L^2(G)}^2
=\frac{1}{2}\|\underline{\bfv}_0\|_{L^2(G)}^2-\int_0^t\int_{G}\underline{\tilde{\bfS}}:\ep(\underline{\bfv})\dxs
-\int_0^t\int_{G}\bfs(\underline{\bfv})\cdot\underline{\bfv}\dxs
\\
&+\int_0^t\int_{G}\underline{\bff}\cdot\underline{\bfv}\dxs
+\int_G\int_0^t\underline{\bfv}\cdot\Phi(\underline{\bfv})\,\dd\underline{\bfW}_\sigma\dx+\frac{1}{2}\int_G\int_0^t\,\dd\Big\langle\int_0^\cdot\Phi(\underline{\bfv})\,\dd\underline{\bfW}\Big\rangle_\sigma\dx.
\end{align*}
Here we applied It\^{o}'s formula to $f(\bfw)=\frac{1}{2}\|\bfw\|_{L^2(G)}^2$. 
Subtracting this from the formula for $\|\underline{\bfv}^N(t)\|_{L^2(G)}^2$ and applying expectation shows
\begin{align*}
&\underline{\E}\bigg[\int_\Q\big(\bfS(\nabla\underline{\bfv}^N)-\bfS(\nabla\underline{\bfv})\big):\nabla\big(\underline{\bfv}^N-\underline{\bfv}\big)\dxs\bigg]\\
&+\underline{\E}\bigg[\int_\Q\big(\bfs(\underline{\bfv}^N)-\bfs(\underline{\bfv})\big):\big(\underline{\bfv}^N-\underline{\bfv}\big)\dxs\bigg]\\
&=\frac{1}{2}\underline{\E}\bigg[-\int_G |\underline{\bfv}^N(T)|^2\dx+\int_G |\underline{\bfv}(T)|^2\dx+\int_G |\mathcal P^N\underline{\bfv}^N_0|^2\dx-\int_G |\underline{\bfv}_0|^2\dx\bigg]\\
&+\underline{\E}\bigg[\int_\Q\big(\underline{\tilde{\bfS}}-\bfS(\ep(\underline{\bfv}^N))\big):\ep(\underline{\bfv})\dxs-\int_\Q\bfS(\ep(\underline{\bfv})):\ep\big(\underline{\bfv}^N-\underline{\bfv}\big)\dxs\bigg]\\
&+\underline{\E}\bigg[\int_\Q\big(\bfs(\underline{\bfv})-\bfs(\underline{\bfv}^N)\big):\underline{\bfv}\dxs-\int_\Q\bfs(\underline{\bfv})\cdot\big(\underline{\bfv}^N-\underline{\bfv}\big)\dxs\bigg]\\
&+\underline{\E}\bigg[\int_\Q\bff\cdot(\underline{\bfv}^N-\underline{\bfv})\dxs
+\frac{1}{2}\int_G\int_0^T \,\dd\Big\langle\int_0^\cdot\Phi(\underline{\bfv}^N)\,\dd\underline{\bfW}^{N,N}\Big\rangle_\sigma\dx\bigg]\\
&-\frac{1}{2}\,\underline{\E}\bigg[\int_G\int_0^T\,\dd\Big\langle\int_0^\cdot\Phi(\underline{\bfv})\,\dd\underline{\bfW}\Big\rangle_\sigma\dx\bigg].
\end{align*}
Letting $N\rightarrow\infty$ shows using (\ref{eq:conv'}) and monotonicity of $\bfS$
\begin{align*}
\lim_N&\bigg(\underline{\E}\bigg[\int_\Q\big(\bfS(\ep(\underline{\bfv}^N))-\bfS(\ep(\underline{\bfv}))\big):\ep\big(\underline{\bfv}^N-\underline{\bfv}\big)\dxt\bigg]\\
&\leq \,\frac{1}{2}\lim_N\underline{\E}\bigg[\int_G\int_0^T \,\dd\bigg(\Big\langle\int_0^\cdot\Phi(\underline{\bfv}^N)\,\dd\underline{\bfW}^{N,N}\Big\rangle_\sigma-\Big\langle\int_0^\cdot\Phi(\underline{\bfv})\,\dd\underline{\bfW}\Big\rangle_\sigma\bigg)\dx\bigg].
\end{align*}
Here we used $\liminf_N\underline{\E}\big[\int_G\big(|\underline{\bfv}^N(T)|^2-|\underline{\bfv}(T)|^2\big)\dxs\big]\geq0$ which follows by lower semi-continuity and weak convergence of $\underline{\bfv}^N(T)$.
On account of (\ref{eq:convWm'})
and (\ref{eq:compactdelta'}) together with (\ref{eq:phi}) we gain for the last integral
\begin{align*}
\underline{\E}\bigg[\int_G\int_0^T \,\dd\Big\langle\int_0^\cdot\Phi(\underline{\bfv}^N)\,\dd\underline{\bfW}^{N,N}\Big\rangle_\sigma\dx\bigg]\longrightarrow\underline{\E}\bigg[\int_G\int_0^T \,\dd\Big\langle\int_0^\cdot\Phi(\underline{\bfv})\,\dd\underline{\bfW}\Big\rangle_\sigma\dx\bigg],\quad N\rightarrow\infty.
\end{align*}
We finally obtain
\begin{align*}
 &\lim_{N\rightarrow\infty}\underline{\E}\bigg[\int_\Q\big(\bfS(\ep(\underline{\bfv}^N))-\bfS(\ep(\underline{\bfv}))\big):\ep\big(\underline{\bfv}^N-\underline{\bfv}\big)\dxt\bigg]=0.
\end{align*}
As a consequence of the monotonicity of $\bfS$ and $\bfs$ we have established
\begin{align*}
\ep(\underline{\bfv}^N)&\longrightarrow \ep(\underline{\bfv})\quad\underline{\p}\otimes\LL^{d+1}\text{-a.e.}.
\end{align*}
This implies (\ref{eq:StildeS'}) and the proof of Theorem \ref{thm:2.2} is hereby complete.
\end{proof}

\begin{remark}\label{rem:ps}
According to the remarks in \cite{IkWa} (beginning of the proof of Thm. 2.7. on p. 9) it is possible to choose the new probability space as
$$(\underline{\Omega},\underline{\F},\underline{\p})=([0,1);\mathcal B[0,1);\mathcal L^1);$$
especially it does not depend on the choice of $\alpha$.
\end{remark}

\begin{corollary}
\label{cor:2.3}
Let the assumptions of Theorem \ref{thm:2.2} be satisfied and in addition 
\begin{equation*}
\int_{L^2_{\Div}(G)}\big\|\bfu\big\|_{L^2(G)}^\beta\,\dd\Lambda_0(\bfu)<\infty,\quad
\int_{L^2(\Q)}\big\|\bfg\big\|_{L^2(\Q)}^\beta\,\dd\Lambda_\bff(\bfg)<\infty,
\end{equation*}
for some $\beta\geq2$. Then
there is a martingale weak solution (\ref{eq:stokes}) such that
\begin{align*}
\E&\bigg[\sup_{t\in(0,T)}\int_G |\underline{\bfv}(t)|^2\dx+\int_\Q |\nabla \underline{\bfv}|^p\dxt+\alpha\int_\Q |\underline{\bfv}|^q\dxt\bigg]^\frac{\beta}{2}\\&\leq c\,\E\bigg[1+\int_{L^2_{\Div}(G)}\big\|\bfu\big\|_{L^2(G)}^2\,\dd\Lambda_0(\bfu)+
\int_{L^2(Q)}\big\|\bfg\big\|_{L^2(Q)}^2\,\dd\Lambda_\bff(\bfg)\dxt\bigg]^\frac{\beta}{2},
\end{align*}
where $c$ is independent of $\alpha$.
\end{corollary}
\begin{proof}
Taking the supremum w.r.t. time and the $\frac{\beta}{2}$-th power of (\ref{eq:2.1neu}) implies
\begin{align*}
\frac{1}{2}&\E\bigg[\sup_{(0,T)}\int_G|\bfv^N(t)|^2\dx\bigg]^\frac{\beta}{2}+\E\bigg[\int_0^T\int_G |\nabla\bfv^N|^p+\alpha|\bfv^N|^q\dxs\bigg]^{\frac{\beta}{2}}\\
&\leq c\,\E\bigg[1+\int_G |\bfv_0|^2\dx+\int_0^T\int_G |\bff||\bfv^N|\dxs\bigg]^\frac{\beta}{2}\\&+c\,\E\bigg[\sup_{(0,T)}\Big|\int_G\int_0^t\bfv^N\cdot\Phi(\bfv^N)\,\dd\bfW^N_\sigma\dx\Big|\bigg]^\frac{\beta}{2}\\
&+c\,\E\bigg[\int_G\int_0^T\,\dd\Big\langle\int_0^\cdot\Phi(\bfv^N)\,\dd\bfW^N\Big\rangle_\sigma\dx\bigg]^\frac{\beta}{2}.
\end{align*}
Obviously it holds
\begin{align*}
\E\bigg[&\int_0^T\int_G |\bff||\bfv^N|\dxs\bigg]^\frac{\beta}{2}\\&\leq c\,\E\bigg[\int_0^T\bigg(\int_G|\bfv^N|^2\dx\bigg)^{\frac{\beta}{2}}\ds\bigg]+c\,\E\bigg[\int_\Q |\bff|^2\dxt\bigg]^\frac{\beta}{2}.
\end{align*}
 Moreover, we have
as a consequence of Burgholder-Davis-Gundi inequality, (\ref{eq:phi}) and Young's inequality (similarly to the proof of Theorem \ref{thm:2.1})
\begin{align*}
\E\bigg[\sup_{t\in(0,T)}|\mathcal K(t)|\bigg]^{\frac{\beta}{2}}&:=\E\bigg[\sup_{t\in(0,T)}\Big|\int_0^t\int_G\bfv^N\cdot\Phi(\bfv^N)\dx\,\dd\bfW^N_\sigma\Big|\bigg]^{\frac{\beta}{2}}\\
&=\E\bigg[\sup_{t\in(0,T)}\Big|\int_0^t\sum_i\int_G\bfv^N\cdot g_i(\bfv^N)\dx\,\dd\beta_i(\sigma)\Big|\bigg]^{\frac{\beta}{2}}\\
&\leq c\,\E\bigg[\int_0^T\sum_i\bigg(\int_G\bfv^N\cdot g_i(\bfv^N)\dx\bigg)^2\dt\bigg]^{\frac{\beta}{4}}\\
&\leq c\,\E\bigg[\bigg(\int_0^T\bigg(\sum_{i=1}^N \int_G|\bfv^N|^2\dx\int_G |g_i(\bfv^N)|^2\dx\bigg)\dt\bigg]^{\frac{\beta}{4}}\\
&\leq c\,\E\bigg[1+\int_0^T\bigg( \int_G|\bfv^N|^2\dx\bigg) ^2\dt\bigg]^{\frac{\beta}{2}}\\
&\leq \delta\,\E\bigg[\sup_{t\in(0,T)}\int_G|\bfv^N|^2\dx\bigg]^{\frac{\beta}{2}}+c(\delta)\,\E\bigg[1+\int_0^T\int_G|\bfv^N|^2\dxt\bigg]^{\frac{\beta}{2}}.
\end{align*}
So we have shown (choosing $\delta$ small enough)
\begin{align*}
&\E\bigg[\sup_{t\in(0,T)}\bigg(\int_G|\bfv^N(t)|^2\dx\bigg)^\frac{\beta}{2}\bigg]+\E\bigg[\int_\Q |\nabla\bfv^N|^p+\alpha|\bfv^N|^q\dxt\bigg]^{\frac{\beta}{2}}\\
&\leq c\,\E\bigg[1+\int_G|\bfv_0|^2\dx+\int_\Q |\bff|^2\dxt+c\,\E\bigg[\int_0^T\bigg(\int_G|\bfv^N|^2\dx\bigg)^{\frac{\beta}{2}}\ds\bigg].
\end{align*}
Gronwall's Lemma implies
\begin{align}\label{eq:aprimpro}
\begin{aligned}
&\E\bigg[\sup_{t\in(0,T)}\int_G|\bfv^N(t)|^2\dx\bigg]^{\frac{\beta}{2}}+\E\bigg[\int_\Q |\nabla\bfv^N|^p+\alpha|\bfv^N|^q\dxt\bigg]^\frac{\beta}{2}\\
&\leq c\,\E\bigg[1+\int_G|\bfv_0|^2\dx+\int_\Q |\bff|^2\dxt\bigg]^\frac{\beta}{2}
\end{aligned}
\end{align}
which gives the claimed inequality.
\end{proof}

\section{Non-stationary flows}
\label{sec:stochmain}
In this section we prove Theorem \ref{thm:main}. The proof is divided in several steps. First we approximate the equation by an equation satisfying the assumptions from the last section. Due to Theorem \ref{thm:2.1} we have a solution to this approximated system. Then we obtain a priori estimates and follow the weak convergence of a subsequence. In the second step we prove compactness of the approximated velocity. In order to pass to the limit in the nonlinear stress deviator we use the $L^\infty$-truncation and monotone operator theory.\\\

\textbf{Step 1: a priori estimates and weak convergence}\\
Let us consider the equation
\begin{align}\label{eq:appr}
\begin{cases}\dd\bfv=&\Div\bfS(\ep(\bfv))\dt-\frac{1}{m}|\bfv|^{q-2}\bfv\dt+\nabla\pi\dt\\&
-\Div\big(\bfv\otimes\bfv\big)\dt+\bff\dt+\Phi(\bfv)\,\dd\bfW_t\\
\bfv(0)=&\bfv_0\end{cases}.
\end{align}
By Theorem \ref{thm:2.1} and Theorem \ref{thm:2.2} for $\alpha=\tfrac{1}{m}$ we know that there  a martingale solution
$$\big((\Omega,\mathcal F,(\mathcal F_t),\p),\bfv^m,\bfv_0^m,\bff^m,\bfW)$$ to (\ref{eq:appr}) with $\bfv^m\in\mathcal V_{p,q}$, $\Lambda_0=\p\circ(\bfv^m_0)^{-1}$ and $\Lambda_\bff=\p\circ(\bff^m)^{-1}$ (we skip the underlines for simplicity). To be precise there holds for all $\bfphi\in C^\infty_{0,\Div}(G)$
\begin{align*}
\int_G&\bfv^m(t)\cdot\bfvarphi\dx +\int_0^t\int_G\bfS(\ep(\bfv^m)):\ep(\bfphi)\dxs+\frac{1}{m}\int_0^t\int_G|\bfv^m|^{q-2}\bfv^m\cdot\bfphi\dxs\\&=\int_0^t\int_G\bfv^m\otimes\bfv^m:\ep(\bfphi)\dxs+\int_G\bfv^m_0\cdot\bfvarphi\dx
+\int_G\int_0^t\bff^m\cdot\bfphi\dxs+\int_G\int_0^t\Phi(\bfv^m)\,\dd\bfW_\sigma\cdot \bfvarphi\dx.
\end{align*}
Note that due to remark \ref{rem:ps} the probability space can be choosen independently of $m$. The same is true for the Brownian motion $\bfW$. Theorem \ref{thm:2.1} yields uniform estimates for $\bfv^m$ in
$$ L^2(\Omega,\F,\p;L^\infty(0,T;L^2(G)))\cap L^p(\Omega,\F,\p;L^p(0,T;W_{0,\Div}^{1,p}(G))).$$
Using Corollary \ref{cor:2.3} and the assumptions on $\Lambda_\bff$ and $\Lambda_0$ in \eqref{initial} we even gain
\begin{align}\label{eq:aprimpro}
\E\bigg[\sup_{t\in(0,T)}\bigg(\int_G|\bfv^m(t)|^2\dx\bigg)^{\frac{\beta}{2}}\bigg]+\E\bigg[\int_\Q |\nabla\bfv^m|^p+\frac{|\bfv^m|^q}{m}\dxt\bigg]^{\frac{\beta}{2}}
\leq c(\beta).
\end{align}
This implies together with a parabolic interpolation and the choice of $\beta$
\begin{align}\label{eq:aprimpro2}
\E\bigg[\int_\Q |\bfv^m|^{r_0}\dxt\bigg]\leq c\quad\text{for all }r_0:=p\frac{d+2}{d}
\end{align}
uniformly in $m$. From this, (\ref{eq:aprimpro}) and the assumption $p>\frac{2d+2}{d+2}$ we gain
\begin{align}\label{eq:aprimpro2'}
\E\bigg[\int_\Q |\bfv^{m}\otimes \bfv^m|^{p_0}+\int_\Q |\nabla\big(\bfv^{m}\otimes \bfv^m\big)|^{p_0}\dxt\bigg]\leq c 
\end{align}
for some $p_0>1$.
We obtain limit functions $\bfv$, $\tilde{\bfS}$, $\bfV$, $\tilde{\Phi}$ such that after passing to subsequences
\begin{align}\label{eq:conv2}
\begin{aligned}
\bfv^m&\rightharpoondown \bfv \quad\text{in}\quad  L^{\frac{\beta}{2}p}(\Omega,\F,\p;L^p(0,T;W^{1,p}_{0,\Div}(G))),\\
\bfv^m&\rightharpoondown \bfv \quad\text{in}\quad L^\beta(\Omega,\F,\p;L^r(0,T;L^2(G)))\quad \forall r<\infty\\
\frac{1}{m}|\bfv^m|^{q-2}\bfv^m&\rightarrow 0 \quad\text{in}\quad L^{\frac{\beta}{2}q'}(\Omega,\F,\p;L^{q'}(\Q)),\\
\bfS(\ep(\bfv^m))&\rightharpoondown \tilde{\bfS} \quad\text{in}\quad L^{p'}(\Omega,\F,\p;L^{p'}(\Q)),\\
\bfS(\ep(\bfv^m))&\rightharpoondown \tilde{\bfS} \quad\text{in}\quad L^{p'}(\Omega,\F,\p;L^{p'}(0,T;W^{-1,p'}(G))),\\
\bfv^{m}\otimes \bfv^m&\rightharpoondown \bfV \quad\text{in}\quad L^{p_0}(\Omega,\F,\p;L^{p_0}(0,T;W^{1,p_0}(G))),\\
\Phi(\bfv^m)&\rightharpoondown \tilde{\Phi} \quad\text{in}\quad L^\beta(\Omega,\F,\p;L^r(0,T;L_2(U,L^2(G))))\quad\forall r<\infty.
\end{aligned}
\end{align}
Moreover, we know
\begin{align*}
 \bfv &\in L^\beta(\Omega,\F,\p;L^\infty(0,T;L^2(G)),\\ \tilde{\Phi} &\in L^\beta(\Omega,\F,\p;L^\infty(0,T;L_2(U,L^2(G)))).
\end{align*}

In order to introduce the pressure we set
\begin{align*}
\bfH_1^m&:=\bfS(\ep(\bfv^m)),\\ \bfH_2^m&:=\nabla\Delta^{-1}\bff^m+\nabla\Delta^{-1}\Big(\frac{1}{m}|\bfv^m|^{q-2}\bfv^m\Big)+\bfv^{m}\otimes\bfv^m
,\\
\Phi^m&:=\Phi(\bfv^m).
\end{align*}
Using Theorem \ref{thm:pi}, Corollary \ref{cor:neu} and Corollary \ref{cor:pi} we obtain functions $\pi^m_h$, $\pi^m_{1}$, $\pi^m_{2}$ adapted to $(\underline{\mathcal F_t})$ and $\Phi^m_\pi$ progressively measurable such that
\begin{align}
\int_G&\big(\bfv^m-\nabla\pi^m_h\big)(t)\cdot\bfvarphi\dx \nonumber\\&=\int_G\bfv^m_0\cdot\bfvarphi\dx
-\int_G\int_0^t\big(\bfH_1^m-\pi_{1}^m I\big):\nabla\bfphi\dxs+\int_G\int_0^t\Div\big(\bfH_2^m-\pi_{2}^m I\big)\cdot\bfphi\dxs\nonumber\\
&+\int_G\int_0^t\Phi^m\,\dd\bfW_\sigma\cdot \bfvarphi\dx+\int_G\int_0^t\Phi_\pi^m\,\dd\bfW_\sigma\cdot \bfvarphi\dx.\label{eq:diffpi0}
\end{align}
The following bounds hold uniformly in $m$\footnote{Here we use the continuity of $\nabla\Delta^{-1}$ from $L^{p_0}(G)$ to $W^{1,p_0}(G)$.}
\begin{align}\label{eq:convu}
\begin{aligned}
\bfH_1^m&\in L^{\frac{\beta}{2}p'}(\Omega,\F,\p,L^{p'}(\Q)),\\
\bfH_2^m&\in L^{p_0}(\Omega,\F,\p,L^{p_0}(0,T;W^{1,p_0}(G))),\\
\Phi^m&\in L^\beta(\Omega,\F,\p,L^\infty(0,T;L_2(U,L^2(G)))).
\end{aligned}
\end{align}
We have the same uniform bounds for the pressure functions
\begin{align}\label{eq:convpi0}
\begin{aligned}
\pi_h^m&\in L^\beta(\Omega,\F,\p,L^\infty(0,T;L^{2}(G))),\\
\pi_{1}^m&\in L^{\frac{\beta}{2}p'}(\Omega,\F,\p,L^{p'}(\Q)),\\
\pi_{2}^m&\in L^{p_0}(\Omega,\F,\p,L^{p_0}(0,T;W^{1,p_0}(G)))),\\
\Phi_\pi^m&\in L^\beta(\Omega,\F,\p,L^\infty(0,T;L_2(U,L^2(G)))).
\end{aligned}
\end{align}
Here we used Corollary \ref{cor:neu} and Corollary \ref{cor:pi}. 
For the harmonic pressure we obtain using regularity theory for harmonic functions and Corollary \ref{cor:pir}
\begin{align}\label{eq:harmpi}
\pi_h^m&\in L^\beta(\Omega,\F,\p;L^r(0,T;W^{k,\infty}_{loc}(G)))
\end{align}
for all $k\in\N$. After passing to subsequences (not relabeled) we have the following convergences (where $r<\infty$ is arbitrary)
 \begin{align}\label{eq:convpi}
\begin{aligned}
\pi_h^m&\rightharpoondown\pi_h\qquad\text{in}\quad L^\beta(\Omega,\F,\p;L^r(0,T;W^{k,r}_{loc}(G))),\\
\pi_{1}^m&\rightharpoondown\pi_{1}\,\qquad\text{in}\quad L^{\frac{\beta}{2}p'}(\Omega,\F,\p,L^{p'}(\Q)),\\
\pi_{2}^m&\rightharpoondown\pi_{2}\,\qquad\text{in}\quad L^{p_0}(\Omega,\F,\p,L^{p_0}(0,T;W^{1,p_0}(G)))),\\
\Phi_\pi^m&\rightharpoondown\Phi_\pi\,\qquad\text{in}\quad L^\beta(\Omega,\F,\p,L^r(0,T;L_2(U,L^2(G)))).
\end{aligned}
\end{align}
In the following we need to show that the limit functions in \eqref{eq:conv2} satisfy $\bfV=\bfv\otimes\bfv$ and $\tilde{\Phi}=\Phi(\bfv)$ hold. We will realize this by compactness arguments and a change of the probability space similar to the proof of Theorem \ref{thm:2.2}. After this, in a final step, we will show $\tilde{\bfS}=\bfS(\ep(\bfv))$.\\\

\textbf{Step 2: compactness}\\
Now we will show compactness of $\bfv^m$ where we follow ideas from \cite{Ho1}, section 4. In order to include the pressure in
the compactness method we have to deal with weak convergences. This situation is not covered by the classical Skorokhod Theorem. However, a generalization of it -- the Jakubowski-Skorokhod Theorem, see \cite{jakubow} -- applies to quasi-polish spaces (also used in \cite{BrHo}). This includes weak topologies of Banach spaces.\\
 First we deduce from (\ref{eq:appr})-(\ref{eq:aprimpro2})
\begin{align*}
\E\bigg[\Big\|\bfv^m(t)-\int_0^t\Phi(\bfv^m)\,\dd\bfW_\sigma\Big\|_{W^{1,p_0}([0,T];W_{\Div}^{-1,p_0}(G))}\bigg]\leq c.
\end{align*}
For the stochastic term we quote from \cite{Ho1}, proof of Lemma 4.6, for some $\mu=\mu(d,p)>0$
\begin{align*}
\E\bigg[\Big\|\int_0^t\Phi(\bfv^m)\,\dd\bfW_\sigma\Big\|_{C^{\mu}([0,T]; L^2(G))}\bigg]\leq c\bigg(1+\int_{\Omega\times \Q}|\bfv^m|^{r_0}\dxt\,\dd\p\bigg)\leq c.
\end{align*}
This is a consequence of (\ref{eq:aprimpro}), with $r_0>2$ and assumption (\ref{eq:phi}). Combining the both informations above shows
\begin{align}\label{eq:vt1}
\E\Big[\|\bfv^m\|_{C^{\mu}([0,T]; W_{\Div}^{-1,p_0}(G))}\Big]\leq c.
\end{align}
and also for some $\lambda>0$
\begin{align}\label{eq:vt2}
\E\Big[\|\bfv^m\|_{W^{\lambda,p_0}(0,T; W_{\Div}^{-1,p_0}(G))}\Big]\leq c.
\end{align}
An interpolation with $L^{p_0}(0,T; W_{0,\Div}^{1,p_0}(G))$ yields on account of (\ref{eq:aprimpro}) for some $\kappa>0$ (see \cite{Am}, Thm. 3.1)
\begin{align}\label{eq:vt3}
\E\Big[\|\bfv^m\|_{W^{\kappa,p_0}(0,T; L^{p_0}_{\Div}(G))}\Big]\leq c.
\end{align}
As a consequence of $p>\tfrac{2d+2}{d+2}$ we have
\begin{align*}
W^{\kappa,p_0}(0,T;L^{p_0}_{\Div}(G))\cap L^p(0,T;W^{1,p}_{0,\Div}(G))\hookrightarrow L^{r}(0,T;L^{r}_{\Div}(G))
\end{align*}
compactly for all $r<p\tfrac{d+2}{d}$. We will use this embedding in order to show compactness of $\bfv^m$. Similarly, we argue for the harmonic pressure $\pi^m_h$. We have compactness of the embedding
\begin{align*}
L^\infty(0,T;L^2(G))\cap\left\{\Delta u(t)=0\,\,\text{for a.e. }t\right\}\hookrightarrow L^r(0,T;L^r_{loc}(G))
\end{align*} which follows from local regularity theory for harmonic functions and Lebesgue's Theorem about dominated convergence (cf. \cite{Wo}, (4.24)).\\
We consider th path space
\begin{align*}
\mathscr V:=& L^r(0,T;L^r_{\Div}(G))\otimes L^r(0,T;L^r_{loc}(G))\otimes\big(L^{p'}(\Q),w\big)\otimes\big(L^{p_0}(0,T;W^{1,p_0}(G))),w\big)\\&\otimes\big(L^r(0,T;L_2(U,L^2(G))),w\big)\otimes C([0,T],U_0)\otimes L^2(G)\otimes L^2(\Q)
\end{align*}
We will use the following notations ($w$ refers to the weak topology):
\begin{itemize}
\item $\nu_{\bfv^m}$ is the law of $\bfv^m$ on $L^r(0,T;L^r(G))$;
\item $\nu_{\pi^m_h}$ is the law of $\pi_h^m$ on $L^r(0,T;L^r_{loc}(G))$;
\item $\nu_{\pi^m_1}$ is the law of $\pi_1^m$ on $\big(L^{p'}(\Q),w\big)$;
\item $\nu_{\pi^m_2}$ is the law of $\pi_2^m$ on $\big(L^{p_0}(0,T;W^{1,p_0}(G))),w\big)$;
\item $\nu_{\Phi^m_\pi}$ is the law of $\Phi_\pi^m$ on $\big(L^r(0,T;L_2(U,L^2(G))),w\big)$;
\item $\nu_\bfW$ is the law of $\bfW$ on $C([0,T],U_0)$, where $U_0$ is defined in (\ref{eq:U0});
\item $\bfnu^m$ is the joint law of $\bfv^m$, $\pi^m_h$, $\pi^m_1$, $\pi^m_2$, $\Phi_\pi^m$, $\bfW$, $\bfv_0$ and $\bff$ on $\mathscr V$.
\end{itemize}
We need to show tightness of the measure $\bfnu^m$.\\
We consider the ball $\mathcal B_R$ in the space $W^{\kappa,p_0}(0,T;L^{p_0}_{\Div}(G))\cap L^p(0,T;W^{1,p}_{\Div}(G))$
and obtain for its complement $\mathcal B_R^C$ by (\ref{eq:aprimpro2}) and (\ref{eq:vt3})
\begin{align*}
\nu_{\bfv^m}&(\mathcal B_R^C)=\p\big(\|\bfv^m\|_{W^{\kappa,p_0}(0,T;L^{p_0}_{\Div}(G))}+\|\bfv^m\|_{ L^p(0,T;W^{1,p}_{\Div}(G))}\geq R\big)\\&
\leq \frac{1}{R}\,\E\Big(\|\bfv^m\|_{W^{\kappa,p_0}(0,T;L^{p_0}_{\Div}(G))}+\|\bfv^m\|_{ L^p(0,T;W^{1,p}_{\Div}(G))}\Big)
\leq \frac{c}{R}.
\end{align*}
So for a fixed $\eta>0$ we find $R(\eta)$ with
\begin{align*}
\nu_{\bfv^m}(\mathcal B_{R(\eta)})&\geq 1-\frac{\eta}{8}.
\end{align*}
Using (\ref{eq:harmpi}) we can show that also the law of $\pi^m_h$ is tight; i.e. there exists a compact set
$C_\pi\subset L^r(0,T;L^r_{loc}(G))$ such that $\nu_{\pi^m_h}(C_\pi)\geq1-\tfrac{\eta}{8}$. Due to the reflexivity of the corresponding spaces
we find compact sets for $\pi_1^m,\pi_2^m$ and $\Phi_\pi^m$ with measures greater or equal then $1-\tfrac{\eta}{8}$.
The law $\nu_{\bfW}$ is tight as it coincides with the law of $\bfW$ which is a Radon measure on the Polish space $C([0,T],U_0)$. So, there exists a compact set
$C_\eta\subset C([0,T],U_0)$ such that $\nu_{\bfW^m}(C_\eta)\geq1-\tfrac{\eta}{8}$. By the same argument we can find compact subsets of $L^2_{\Div}(G)$ and $L^2(\Q)$ such that their measures ($\Lambda_0$ and $\Lambda_{\bff}$) are smaller than $1-\frac{\eta}{8}$.
Hence, we can find a compact subset 
$\mathscr V_\eta\subset \mathscr V$
such that $\bfnu^m(\mathscr V_\eta)\geq1-\eta$. Thus, $\left\{\bfnu^m,\,\,m\in\N\right\}$ is tight in the same space.
On account of the Jakubowski-Skorohod Theorem \cite{jakubow}
we can infer the existence of a probability space $(\underline{\Omega},\underline{\F},\underline{\p})$, a sequence $(\underline{\bfv}^N,\underline{\pi}_h^m,\underline{\pi}_1^m,\underline{\pi}_2^m,\underline{\Phi}_\pi^m,\underline{\bfW}^m,\underline{\bfv}_0^m,\underline{\bff}^m)$ and $(\underline{\bfv},\underline{\pi}_h,\underline{\pi}_1,\underline{\pi}_2,\underline{\Phi}_\pi,\underline{\bfW},\underline{\bfv}_0,\underline{\bff})$ on $(\underline{\Omega},\underline{\F},\underline{\p})$ both with values in $\mathscr V$ such that the following holds
\begin{itemize}
\item The laws of $(\underline{\bfv}^m,\underline{\pi}_h^m,\underline{\pi}_1^m,\underline{\pi}_2^m,\underline{\Phi}_\pi^m,\underline{\bfW}^m,\underline{\bfv}_0^m,\underline{\bff}^m)$ and $(\underline{\bfv},\underline{\pi}_h,\underline{\pi}_1,\underline{\pi}_2,\underline{\Phi}_\pi,\underline{\bfW},\underline{\bfv}_0,\underline{\bff})$ under $\underline{\p}$ coincide with $\bfnu^m$ and $\bfnu:=\lim_m\bfnu^m$.
\item We have $\underline{\p}$-a.s. the weak convergences 
\begin{align*}
\underline{\pi}_1^m&\rightharpoondown \underline{\pi}_1\quad\,\text{in}\,\quad L^{p'}(\Q),\\
\underline{\pi}_2^m&\rightharpoondown \underline{\pi}_2\quad\,\text{in}\,\quad L^{p_0}(0,T;W^{1,p_0}(G))),\\
\underline{\Phi}_\pi^m&\rightharpoondown \underline{\Phi}_\pi\quad\,\text{in}\,\quad L^r(0,T;L_2(U,L^2(G))).
\end{align*}
\item We have $\underline{\p}$-a.s. the strong convergences
\begin{align*}
\underline{\bfv}^m&\rightarrow \underline{\bfv}\,\,\,\,\quad\text{in}\quad L^r(0,T;L^r(G)),\\
\underline{\pi}_h^m&\rightarrow \underline{\pi}_h\quad\,\text{in}\,\quad L^r(0,T;L^r_{loc}(G)),\\
\underline{\bfW}^m&\rightarrow \underline{\bfW}\quad\,\text{in}\,\quad C([0,T],U_0),\\
\underline{\bfv}_0^m&\rightarrow \underline{\bfv}_0\quad\,\,\text{in}\quad L^2(G),\\
\underline{\bff}^m&\rightarrow \underline{\bff}\qquad\text{in}\quad L^2(0,T;L^2(G)).
\end{align*}
\item 
We have for all $\alpha<\infty$
\begin{align*}
\int_{\underline{\Omega}}\bigg(\sup_{[0,T]}\|\underline{\bfW}^m(t)\|_{U_0}^\alpha\bigg)\,\dd\underline{\p}=\int_{\Omega}\bigg(\sup_{[0,T]}\|\bfW(t)\|_{U_0}^\alpha\bigg)\,\dd\p.
\end{align*}
\end{itemize}
On account of the equality of laws we gain the weak convergences
\begin{align*}
\underline{\pi}_1^m&\rightharpoondown \underline{\pi}_1\quad\text{in}\quad L^{p'}(\underline{\Omega},\underline{\F},\underline{\p}, L^{p'}(\Q),\\
\underline{\pi}_2^m&\rightharpoondown \underline{\pi}_2\quad\text{in}\quad L^{p_0}(\underline{\Omega},\underline{\F},\underline{\p}, L^{p_0}(0,T;W^{1,p_0}(G))),\\
\underline{\Phi}_\pi^m&\rightharpoondown \underline{\Phi}_\pi\quad\text{in}\quad L^{p_0}(\underline{\Omega},\underline{\F},\underline{\p}, L^r(0,T;L_2(U,L^2(G)))),
\end{align*}
after choosing a subsequence, and by Vitali's convergence Theorem the strong convergences
\begin{align}\label{eq:convWm}
\underline{\bfW}^m&\longrightarrow\underline{\bfW}\quad\text{in}\quad L^2(\underline{\Omega},\underline{\F},\underline{\p};C([0,T],U_0)),\\\label{eq:compactdelta}
\underline{\bfv}^m&\longrightarrow \underline{\bfv}\quad\text{in}\quad L^r(\underline{\Omega}\times \Q;\underline{\p}\otimes\mathcal L^{d+1}),\\
\label{eq:compactpi}
\nabla^k\underline{\pi}_h^m&\longrightarrow \nabla^k\underline{\pi}_h\quad\text{in}\quad L^r(\underline{\Omega}\times(0,T)\times G' ;\underline{\p}\otimes\mathcal L^{d+1}),\\
\label{eq:compactv0}
\underline{\bfv}^m_0&\longrightarrow \underline{\bfv}_0\quad\text{in}\quad L^2(\underline{\Omega}\times G,\underline{\p}\otimes\mathcal L^{d+1}),\\
\label{eq:compactf}
\underline{\bff}^m&\longrightarrow \underline{\bff}\quad\text{in}\quad L^2(\underline{\Omega}\times \Q,\underline{\p}\otimes\mathcal L^{d+1}),
\end{align}
for all $r<p\tfrac{d+2}{d}$ and all $G'\Subset G$. For the harmonic pressure we used local regularity theory for harmonic maps.
This implies for all $\alpha<\infty$
\begin{align}\label{eq:conv_}
\begin{aligned}
\underline{\bfv}^{m}\otimes\underline{\bfv}^m&\rightharpoondown \underline{\bfv}\otimes\underline{\bfv}\quad\text{in}\quad L^{p_0}(\underline{\Omega},\underline{\F},\underline{\p}, L^{p_0}(0,T;W^{1,p_0}(G))),\\
\Phi(\underline{\bfv}^m)&\rightharpoondown \Phi(\underline{\bfv}) \quad\,\,\,\text{in}\quad  L^\beta(\underline{\Omega},\underline{\F},\underline{\p},L^\alpha(0,T;L_2(U,L^2(G)))),\\
\Phi_\pi(\underline{\bfv}^m)&\rightharpoondown \Phi_\pi(\underline{\bfv}) \quad\text{in}\quad  L^\beta(\underline{\Omega},\underline{\F},\underline{\p},L^\alpha(0,T;L_2(U,L^2(G)))).
\end{aligned}
\end{align}
Again we define $(\underline{\mathcal F}_t)$ be the $\underline{\p}$-augmented canonical filtration of the process $(\underline{\bfv},\underline{\pi}_h,\underline{\pi}_1,\underline{\pi}_2,\underline{\Phi}_\pi,\underline{\bfW},\underline{\bff})$, respectively, that is
\begin{equation*}
\begin{split}
\underline{\mathcal F}_t&=\sigma\Big(\sigma\big(\bfr_t\underline{\bfv},\bfr_t\underline{\pi}_h,\bfr_t\underline{\pi}_1,\bfr_t\underline{\pi}_2,\bfr_t\underline{\Phi}_\pi,\bfr_t\underline{\bfW},\bfr_t\underline{\bff})\cup\big\{N\in\underline{\mathcal F};\;\underline{\p}(N)=0\big\}\Big),\quad t\in[0,T].
\end{split}
\end{equation*}
As done in the proof of Theorem \ref{thm:2.2} (but using test-functions from $C_0^\infty(G)$ instead of $C_{0,\Div}^\infty(G)$) we can show that
the equation also hold on the new probability space, i.e. we have $\underline{\p}\otimes\mathcal L^1$-a.e.
 for all $\bfphi\in C^\infty_{0}(G)$
\begin{align*}
\int_G&\big(\underline\bfv^m-\nabla\underline\pi^m_h\big)(t)\cdot\bfvarphi\dx \nonumber\\&=\int_G\underline\bfv^m_0\cdot\bfvarphi\dx
-\int_G\int_0^t\big(\underline\bfH_1^m-\underline\pi_{1}^m I\big):\nabla\bfphi\dxs+\int_G\int_0^t\Div\big(\underline\bfH_2^m-\underline\pi_{2}^m I\big)\cdot\bfphi\dxs\nonumber\\
&+\int_G\int_0^t\Phi(\underline\bfv_m)\,\dd\underline\bfW_\sigma\cdot \bfvarphi\dx+\int_G\int_0^t\underline\Phi_\pi^m\,\dd\underline\bfW_\sigma\cdot \bfvarphi\dx
\end{align*}
using the abbreviations
\begin{align*}
\underline\bfH_1^m&:=\bfS(\ep(\underline\bfv^m)),\\ \underline\bfH_2^m&:=\nabla\Delta^{-1}\underline\bff^m+\nabla\Delta^{-1}\Big(\frac{1}{m}|\underline\bfv^m|^{q-2}\underline\bfv^m\Big)+\bfv^{m}\otimes\bfv^m.
\end{align*}
From the convergences above we gain the limit equation (using again \cite{DeGlTe}, Lemma 2.1, for the convergence of the stochastic integral)
\begin{align}\label{eq:limitvdelta_}
\begin{aligned}
\int_G&\big(\underline\bfv-\nabla\underline\pi_h\big)(t)\cdot\bfvarphi\dx \\&=\int_G\underline\bfv_0\cdot\bfvarphi\dx
-\int_G\int_0^t\big(\underline\bfH_1-\underline\pi_{1} I\big):\nabla\bfphi\dxs+\int_G\int_0^t\Div\big(\underline\bfH_2-\underline\pi_{2} I\big)\cdot\bfphi\dxs\\
&+\int_G\int_0^t\Phi(\underline\bfv)\,\dd\underline\bfW_\sigma\cdot \bfvarphi\dx+\int_G\int_0^t\underline\Phi_\pi\,\dd\underline\bfW_\sigma\cdot \bfvarphi\dx
\end{aligned}
\end{align}
for all $\bfvarphi\in C^\infty_{0}(G)$, where
\begin{align*}
\underline\bfH_1&:=\underline{\tilde\bfS},\quad \underline\bfH_2:=\nabla\Delta^{-1}\underline\bff+\bfv\otimes\bfv.
\end{align*}
It remains to show $\underline{\tilde\bfS}=\bfS(\ep(\underline\bfv))$.
Now we let 
\begin{align*}
\underline{\bfG}_1^m&:=\bfS(\ep(\underline{\bfv}^m))-\underline{\tilde{\bfS}},\\ \underline{\bfG}_2^m&:=\nabla\Delta^{-1}\big(\underline{\bff}^m-\underline{\bff}\big)+\nabla\Delta^{-1}\Big(\frac{1}{m}|\underline{\bfv}^m|^{q-2}\underline{\bfv}^m\Big)+\underline{\bfv}^{m}\otimes\underline{\bfv}^m
-\underline{\bfv}\otimes\underline{\bfv},\\
\underline{\Phi}^m&:=\big(\Phi(\underline{\bfv}^m),-\Phi(\underline{\bfv})\big),\quad \underline{\Phi}_\vartheta^m:=\big(\Phi_\pi(\underline{\bfv}^m),-\Phi_\pi(\underline{\bfv})\big),\\
\underline{\vartheta}^m_{h}&=\underline{\pi}^m_{h}-\underline{\pi}_{h},\quad
\underline{\vartheta}^m_{1}=\underline{\pi}^m_{1}-\underline{\pi}_{1},\quad
\underline{\vartheta}^m_{2}=\underline{\pi}^m_{2}-\underline{\pi}_{2}.
\end{align*}
The following convergences are true
\begin{align}\label{eq:convu}
\begin{aligned}
\underline{\bfv}^m-\underline{\bfv}&\rightharpoondown 0\quad\text{in}\quad  L^{\frac{\beta}{2}p}(\underline{\Omega},\underline{\F},\underline{\p},L^p(0,T;W_0^{1,p}(G))),\\
\underline{\bfv}^m-\underline{\bfv}&\rightharpoondown 0\quad\text{in}\quad  L^\beta(\underline{\Omega},\underline{\F},\underline{\p},L^r(0,T;L^{2}(G)))\quad\forall r<\infty,\\
\underline{\bfG}_1^m&\rightharpoondown 0\quad\text{in}\quad  L^{\frac{\beta}{2}p'}(\underline{\Omega},\underline{\F},\underline{\p},L^{p'}(\Q)),\\
\underline{\bfG}_2^m&\rightharpoondown 0\quad\text{in}\quad  L^{p_0}(\underline{\Omega},\underline{\F},\underline{\p},L^{p_0}(0,T;W^{1,p_0}(G))),\\
\underline{\Phi}^m-\underline{\Phi}&\rightharpoondown0 \quad\text{in}\quad  L^\beta(\underline{\Omega},\underline{\F},\underline{\p},L^r(0,T;L_2(U,L^2(G))))\quad\forall r<\infty.
\end{aligned}
\end{align}
where $\underline{\Phi}=(\Phi(\underline{\bfv}),-\Phi(\underline{\bfv}))$. We have the same convergences for the pressure functions:
\begin{align}\label{eq:convpi}
\begin{aligned}
\underline{\vartheta}_h^m&\rightarrow0 \quad\text{in}\quad L^\beta(\underline{\Omega},\underline{\F},\underline{\p},L^r(0,T;W^{k,r}_{loc}(G)))\quad \forall r<\infty,\\
\underline{\vartheta}_{1}^m&\rightharpoondown 0\quad\text{in}\quad  L^{\frac{\beta}{2}p'}(\underline{\Omega},\underline{\F},\underline{\p},L^{p'}(\Q)),\\
\underline{\vartheta}_{2}^m&\rightharpoondown 0\quad\text{in}\quad L^{p_0}(\underline{\Omega},\underline{\F},\underline{\p},L^{p_0}(0,T;W^{1,p_0}(G))),\\
\underline{\Phi}_\vartheta^m-\underline{\Phi}_\vartheta&\rightharpoondown 0\quad\text{in}\quad  L^\beta(\underline{\Omega},\underline{\F},\underline{\p},L^r(0,T;L_2(U,L^2(G))))\quad\forall r<\infty.
\end{aligned}
\end{align}
Moreover, we have 
\begin{align}\label{eq:harmpi*}
\begin{aligned}
\underline{\vartheta}_h^m&\in L^\beta(\underline{\Omega},\underline{\F},\underline{\p},L^\infty(0,T;L^{2}(G))),\\
\underline{\Phi}^m&\in  L^\beta(\underline{\Omega},\underline{\F},\underline{\p},L^\infty(0,T;L_2(U,L^2(G)))),\\
\underline{\Phi}_\vartheta^m&\in  L^\beta(\underline{\Omega},\underline{\F},\underline{\p},L^\infty(0,T;L_2(U,L^2(G)))),
\end{aligned}
\end{align}
uniformly in $m$.\\
The difference of approximated equation and limit equation reads as
\begin{align}
\int_G&\big(\underline{\bfv}^m-\underline{\bfv}+\nabla\underline{\vartheta}^m_h\big)(t)\cdot\bfvarphi\dx \nonumber\\&=\int_G\big(\underline{\bfv}^m(0)-\underline{\bfv}_0\big)\cdot\bfvarphi\dx
-\int_G\int_0^t\big(\underline{\bfG}_1^m-\underline{\vartheta}_{1}^m I\big):\nabla\bfphi\dxs+\int_G\int_0^t\Div\big(\underline{\bfG}_2^m-\underline{\vartheta}_{2}^m I\big)\cdot\bfphi\dxs\nonumber\\
&+\int_G\int_0^t\underline{\Phi}^m\,\dd(\underline{\bfW}^m_\sigma,\underline{\bfW}_\sigma)\cdot \bfvarphi\dx+\int_G\int_0^t\underline{\Phi}_\vartheta^m\,\dd(\underline{\bfW}^m_\sigma,\underline{\bfW}_\sigma)\cdot \bfvarphi\dx.\label{eq:diffpi0}
\end{align}
for all $\bfvarphi\in C^\infty_{0}(G)$.
In the following we will show that $\underline{\tilde{\bfS}}=\bfS(\ep(\underline{\bfv}))$ holds which will finish the proof of Theorem \ref{thm:main}. Therefore we introduce the sequence $\underline{\bfu}^{m}:=\underline{\bfv}^{m}-\nabla\underline{\vartheta}^{m}_h$ and the double sequence $\underline{\bfu}^{m,k}:=\underline{\bfu}^{m}-\underline{\bfu}^{k}$, $m\geq k$,
for which we have the convergences
\begin{align}\label{eq:convu1}
\underline{\bfu}^m&\rightharpoondown 0\quad\text{in}\quad  L^{p}(\underline{\Omega},\underline{\F},\underline{\p},L^p(0,T;W_0^{1,p}(G))),\\
\underline{\bfu}^m&\rightarrow 0\quad\text{in}\quad  L^r(\underline{\Omega}\times(0,T)\times G';\underline{\p}\otimes\mathcal L^{d+1}),\label{eq:convu2}
\end{align}
and the equation
\begin{align}
\int_G\underline{\bfu}^{m,k}(t)\cdot\bfvarphi\dx&=\int_G\underline{\bfu}^{m,k}_0\cdot\bfvarphi\dx \nonumber\\&-\int_G\int_0^t\big(\underline{\bfG}_1^{m,k}-\underline{\vartheta}_{1}^{m,k} I\big):\nabla\bfphi\dxs+\int_G\int_0^t\Div\big(\underline{\bfG}_2^{m,k}-\underline{\vartheta}_{2}^{m,k} I\big)\cdot\bfphi\dxs\nonumber\\
&+\int_G\int_0^t\underline{\Phi}^{m,k}\,\dd(\underline{\bfW}^m_\sigma,\underline{\bfW}^k_\sigma)\cdot \bfvarphi\dx+\int_G\int_0^t\underline{\Phi}_\vartheta^{m,k}\,\dd(\underline{\bfW}^m_\sigma,\underline{\bfW}^k_\sigma)\cdot \bfvarphi\dx\label{eq:diffpi}
\end{align}
for all $\bfvarphi\in C^\infty_{0}(G)$.
All involved quantities with superscript $^{m,k}$ are defined analogously to $\underline{\bfu}^{m,k}$ by taking an appropriate difference.
\\\

\textbf{Step 3: monotone operator theory and $L^\infty$-truncation}\\
By density arguments we are allowed to test with $\bfvarphi\in W_0^{1,p}\cap L^\infty(G)$.
Since the function $\underline{\bfv}(\omega,t,\cdot)$ does not belong to this class, the $L^\infty$-truncation was used for the deterministic problem (see \cite{FrMaSt} for the steady case and \cite{Wo} for the unsteady problem).
We will apply a variant of it adapted to the stochastic fashion.\\
We define $h_L$ and $H_L$, $L\in\N_0$, by
\begin{align*}
h_L(s)&:=\int_0^s \Psi_L(\theta)\theta\,\mathrm{d}\theta,\quad H_L(\bfxi):=h_L(|\bfxi|),\\
\Psi_L&:=\sum_{\ell=1}^L \psi_{2^{-\ell}},\quad \psi_\delta(s):=\psi(\delta s),
\end{align*}
where $\psi\in C^\infty_0([0,2])$, $0\leq\psi\leq1$, $\psi\equiv1$ on $[0,1]$ and $0\leq-\psi'\leq2$. Now we consider for $\eta\in C^\infty_0(G)$ the function
\begin{align*}
f_L(\bfv):=\int_G \eta H_L(\bfv)\dx
\end{align*}
and apply It\^{o}'s formula (see Lemma \ref{lems:Ito}).
This yields
\begin{align*}
\int_G \eta H_L(\underline{\bfu}^{m,k}(t))\dx=&f_L(\underline{\bfu}^{m,k}(0))+\int_0^t f'_L(\underline{\bfu}^{m,k})\,\dd\underline{\bfu}^{m,k}+\frac{1}{2}\int_0^t f_L''(\underline{\bfu}^{m,k})\,\dd\langle\underline{\bfu}^{m,k}\rangle_\sigma\\
=&\int_G \eta H_L(\underline{\bfu}^m_0-\underline{\bfu}_0^k)\dx\\
&-\int_G\int_0^t\eta\big(\underline{\bfG}_1^{m,k}-\underline{\vartheta}^{m,k}_{1}I\big)
:\nabla\big(\Psi_L(|\underline{\bfu}^{m,k}|)\underline{\bfu}^{m,k}\big)\dxs\\
&-\int_G\int_0^t\big(\underline{\bfG}_1^{m,k}-\underline{\vartheta}^{m,k}_{1}I\big)
:\nabla\eta\otimes\Psi_L(|\underline{\bfu}^{m,k}|)\underline{\bfu}^{m,k}\dxs\\
&+\int_G\int_0^t\eta\Psi_L(|\underline{\bfu}^{m,k}|)\Div\big(\underline{\bfG}_2^{m,k}-\underline{\vartheta}^{m,k}_{2}I\big)\cdot\underline{\bfu}^m\dxs\\
&+\int_G\int_0^t\eta\Psi_L(|\underline{\bfu}^{m,k}|)\underline{\bfu}^{m,k}\cdot\big(\Phi(\underline{\bfv}^{m,k})\,\dd\underline{\bfW}^m_\sigma-\Phi(\underline{\bfv}^k)\,\dd\underline{\bfW}^k_\sigma\big)\dx\\
&+\int_G\int_0^t\eta\Psi_L(|\underline{\bfu}^{m,k}|)\underline{\bfu}^{m,k}\cdot\big(\Phi_\vartheta(\underline{\bfv}^{m,k})\,\dd\underline{\bfW}^m_\sigma-\Phi_\vartheta(\underline{\bfv}^k)\,\dd\underline{\bfW}^k_\sigma\big)\dx\\
&+\frac{1}{2}\int_G\int_0^t\eta D^2H_L(\underline{\bfu}^{m,k})\,\dd\Big\langle\int_0^\cdot\Phi(\underline{\bfv}^{m})\,\dd\underline{\bfW}^m-\int_0^\cdot\Phi(\underline{\bfv^k})\,\dd\underline{\bfW}^k\Big\rangle_\sigma\dx\\
&+\frac{1}{2}\int_G\int_0^t\eta D^2H_L(\underline{\bfu}^{m,k})\,\dd\Big\langle\int_0^\cdot\Phi_\vartheta(\underline{\bfv}^{m})\,\dd\underline{\bfW}^m-\int_0^\cdot\Phi_\vartheta(\underline{\bfv}^k)\,\dd\underline{\bfW}^k\Big\rangle_\sigma\dx\\
=:&(O)+(I)+(II)+(III)+(IV)+(V)+(VI)+(VII).
\end{align*}
Equation (\ref{eq:compactv0}) and $\underline{\bfu}^{m}(0)-\underline{\bfu}^k(0)=\underline{\bfv}^m(0)-\underline{\bfv}^k(0)$ (see Theorem \ref{thm:pi} b) imply that $\underline{\E}\big[(O)\big]\rightarrow0$ if $m,k \rightarrow\infty$.
The aim of the following observations is to show that the expectation values of  $(II)-(VI)$ vanish for $m,k\rightarrow\infty$ which gives the same for $(I)$. By monotone operator theory this yields $\ep(\bfv^m)\rightarrow\ep(\bfv)$ a.e. Although the rough ideas are clear their rigorous proof is quite technical.\\
By construction of $\Psi_L$ we obtain, after passing to a subsequence,
\begin{align}\label{eq:convpsiu}
\Psi_L(|\underline{\bfu}^{m,k}|)\underline{\bfu}^{m,k}&\longrightarrow 0\quad\text{in}\quad L^r(\underline{\Omega}\times \Q,\underline{\p}\times \mathcal L^{d+1}),\quad m,k\rightarrow0,
\end{align}
for all $r<\infty$ (first, we have boundedness in $L^r$, then the strong convergence follows in combination with (\ref{eq:compactdelta})). This implies
\begin{align*}
\underline{\E}\big[(II)\big],\,\underline{\E}\big[(III)\big]
\longrightarrow0,\quad m,k\rightarrow\infty,
\end{align*}
as a consequence of (\ref{eq:convu}) and (\ref{eq:convpi}). 
Since $\underline{\E}[(IV)]=\underline{\E}[(V)]=0$ only $(VI)$ and $(VII)$ remain. We gain as $|D^2 H_L|\leq c(L)$ that
\begin{align*}
(VI)&\leq \,c\sum_{\ell=1}^d\,
\int_G\int_0^t\,\dd\, \Big\langle\int_0^\cdot\big(\Phi(\underline{\bfv}^m)-\Phi(\underline{\bfv}^k)\big)\,\dd\underline{\bfW}^m\Big\rangle^{\ell\ell}_\sigma\dx\\
&+c\sum_{\ell=1}^d\,\int_G\int_0^t \,\dd\,\Big\langle\int_0^\cdot\Phi(\underline{\bfv}^k)\,\dd\big(\underline{\bfW}^m-\underline{\bfW}^k\big)
\Big\rangle^{\ell\ell}_\sigma\dx\\
&+c\sum_{\ell=1}^d\,\int_G\int_0^t \,\dd\,\Big\langle\int_0^\cdot\big(\Phi(\underline{\bfv}^m)-\Phi(\underline{\bfv}^k)\big)\,\dd\underline{\bfW}^m\big),
\int_0^\cdot\big(\Phi(\underline{\bfv}^k)\,\dd\big(\underline{\bfW}^m-\underline{\bfW}^k\big)
\Big\rangle^{\ell\ell}_\sigma\dx\\
&\leq
\,c\sum_{\ell=1}^d\,\int_G\int_0^t\,\dd\, \Big\langle\int_0^\cdot\big(\Phi(\underline{\bfv}^m)-\Phi(\underline{\bfv}^k)\big)\,\dd\underline{\bfW}^m\Big\rangle^{\ell\ell}_\sigma\dx\\
&+c\sum_{\ell=1}^d\,\int_G\int_0^t\,\dd\, \Big\langle\int_0^\cdot\Phi(\underline{\bfv}^k)\,\dd\big(\underline{\bfW}^m-\underline{\bfW}^k\big)
\Big\rangle^{\ell\ell}_\sigma\dx\\
&=:c(VI)_1+c(VI)_2.
\end{align*}
We have by (\ref{eq:phi}) and (\ref{eq:compactdelta})
\begin{align*}
\underline{\E}\big[(VI)_1\big]
&\leq c\,\underline{\E}\bigg[\int_0^t\|\Phi(\underline{\bfv}^m)-\Phi(\underline{\bfv}^k)\|_{L_2(U,L^2(G))}^2\ds\bigg]\\
&\leq c\,\underline{\E}\bigg[\int_0^t\int_G|\underline{\bfv}^m-\underline{\bfv}^k|^2\dxs\bigg]\longrightarrow0,\quad m,k\rightarrow0.
\end{align*}
Moreover, since $\underline{\bfv}^k\in L^2(\underline{\Omega}\times \Q,\underline{\p}\otimes\mathcal L^{d+1})$ uniformly in $k$ we obtain by (\ref{eq:phi2}) and (\ref{eq:convWm})
\begin{align*}
\underline{\E}\big[(VI)_2\big]&=\,\underline{\E}\bigg[\int_0^T\sum_i\bigg(\int_G|g_i(\underline{\bfv}^k)|^2\Var\Big(\underline{\beta}_i^m(1)-\underline{\beta}_i^k(1)\Big)\dx\bigg)\dt\bigg]\\
&\leq\,\underline{\E}\bigg[\int_0^T\bigg(\int_G\sup_i i^2|g_i(\underline{\bfv}^k)|^2\dx\bigg)\dt\bigg]\sum_i\frac{1}{i^2}\Var\Big(\underline{\beta}_i^m(1)-\underline{\beta}_i^k(1)\Big)\\
&\leq c\,\underline{\E}\bigg[\int_0^T\int_G \big(1+|\underline{\bfv}^k|^2\big)\dxt\bigg]\underline{\E}\Big[\|\underline{\bfW}^m-\underline{\bfW}^k\|^2_{C([0,T],U_0)}\Big]\\
&\longrightarrow 0,\quad m,k\rightarrow\infty.
\end{align*}
As a consequence of Corollary \ref{cor:neu} (and the usage of the cut-off function $\eta$) we know that $\Phi_\vartheta$ inherits the properties of $\Phi$, so $(VII)$ can be estimated following
the same ideas.
Plugging all together, we have shown
\begin{align}\label{eq:monop1}
\begin{aligned}
\limsup_{m,k}&\,\underline{\E}\bigg[\int_\Q\eta  \big(\bfS(\ep(\underline{\bfv}^m))-\bfS(\ep(\underline{\bfv}^k))\big):\Psi_L(|\underline{\bfu}^{m,k}|)\ep(\underline{\bfu}^{m,k})\dxs\bigg]\\
&\leq \limsup_{m,k}\underline{\E}\bigg[\int_\Q \eta  \big(\bfS(\ep(\underline{\bfv}^m))-\bfS(\ep(\underline{\bfv}^k))\big):\nabla\big\{\Psi_L(|\underline{\bfu}^{m,k}|)\big\}\otimes\underline{\bfu}^{m,k}\dxs\bigg]\\
&+ \limsup_{m,k}\underline{\E}\bigg[\int_\Q \eta\,\underline{\vartheta}^{m,k}_{1}\Div\big(\Psi_L(|\underline{\bfu}^{m,k}|)\underline{\bfu}^{m,k}\big)\dxs\bigg].
\end{aligned}
\end{align}
Now we want to show that the r.h.s. is bounded in $L$. Since $\Div\underline{u}^{m,k}=0$ there holds
\begin{align*}
\limsup_{m,k}\underline{\E}&\bigg[\int_\Q \eta\,\underline{\vartheta}^{m,k}_{1}\Div\big(\Psi_L(|\underline{\bfu}^{m,k}|)\underline{\bfu}^{m,k}\big)\dxs\bigg]\\&=\limsup_{m,k}\underline{\E}\bigg[\int_\Q \eta\,\underline{\vartheta}^{m,k}_{1}\nabla\big\{\Psi_L(|\underline{\bfu}^{m,k}|)\big\}\cdot\underline{\bfu}^{m,k}\dxs\bigg].
\end{align*}
So, by (\ref{eq:convu}) and (\ref{eq:convpi}), we only need to show
\begin{align}\label{eq:Psip}
\nabla\Psi_L(|\underline{\bfu}^{m,k}|)\underline{\bfu}^{m,k}\in L^p(\underline{\Omega}\times \Q,\underline{\p}\otimes\mathcal L^{d+1})
\end{align}
uniformly in $L$, $m$ and $k$ to conclude
\begin{align}\label{eq:monop2}
\limsup_{m,k}&\,\underline{\E}\bigg[\int_\Q\eta  \big(\bfS(\ep(\underline{\bfv}^m))-\bfS(\ep(\underline{\bfv}^k))\big):\Psi_L(|\underline{\bfu}^{m,k}|)\ep(\underline{\bfu}^{m,k})\dxs\bigg]\leq K.
\end{align}
We have for all $\ell\in\N_0$
\begin{align*}
\big|\nabla\big\{\psi_{2^{-\ell}}(|\underline{\bfu}^{m,k}|)\big\}\underline{\bfu}^{m,k}\big|&\leq \big|\psi'_{2^{-\ell}}(|\underline{\bfu}^{m,k}|)\underline{\bfu}^{m,k}\otimes\nabla\underline{\bfu}^{m,k}\big|\\
&\leq -2^{-\ell}|\underline{\bfu}^{m,k}|\psi'(2^{-\ell}|\underline{\bfu}^{m,k}|)|\nabla\underline{\bfu}^{m,k}|\\
&\leq c |\nabla\underline{\bfu}^{m,k}|\chi_{A_\ell},\\
A_\ell&:=\left\{2^{\ell}< |\underline{\bfu}^{m,k}|\leq 2^{\ell+1}\right\}.
\end{align*}
This implies
\begin{align*}
\big|\nabla\Psi_{L}(|\underline{\bfu}^{m,k}|)\underline{\bfu}^{m,k}\big|&\leq \sum_{\ell=0}^L \big|\nabla\big\{\psi_{2^{-\ell}}(|\underline{\bfu}^{m,k}|)\big\}\underline{\bfu}^{m,k}\big|\\
&\leq c\sum_{\ell=0}^L \big|\nabla\underline{\bfu}^{m,k}\big|\chi_{A_\ell}\leq c|\nabla\underline{\bfu}^{m,k}|.
\end{align*}
This yields (\ref{eq:Psip}) and hence also (\ref{eq:monop2}) is shown.
Now we consider
\begin{align*}
\Sigma_{L,m,k}:=\,\underline{\E}\bigg[\int_\Q\eta  \big(\bfS(\ep(\underline{\bfv}^m))-\bfS(\ep(\underline{\bfv}^k))\big):\Psi_L(|\underline{\bfu}^{m,k}|)\ep(\underline{\bfu}^{m,k})\dxs\bigg].
\end{align*}
On account of (\ref{eq:monop2}) we have $\Sigma_{L,m}\leq K$ independent from $L$ and $m$. Thus, using Cantor's diagonalizing principle we gain a subsequence with
\begin{align*}
\sigma_{\ell,m_l,k_l}:=\,\underline{\E}\bigg[\int_\Q\eta  \big(\bfS(\ep(\underline{\bfv}^{m_l}))-\bfS(\ep(\underline{\bfv}^{k_l}))\big):\psi_{2^{-\ell}}(|\underline{\bfu}^{m_l,k_l}|)\ep(\underline{\bfu}^{m_l,k_l})\dxs\bigg]
\longrightarrow \sigma_\ell,\quad \ell\in\N_0,
\end{align*}
for $l\rightarrow\infty$. We know as a consequence of the monotonicity of $\bfS$ that $\sigma_\ell\geq0$ for all $\ell\in\N$.
Moreover, $\sigma_\ell$ is increasing in $\ell$. This implies on account of (\ref{eq:monop2})
\begin{align*}
0\leq \sigma_0\leq \frac{\sigma_0+\sigma_1+...+\sigma_\ell}{\ell}\leq \frac{K}{\ell}
\end{align*}
for all $\ell\in\N$. Hence we have $\sigma_0=0$ and therefore
\begin{align*}
\underline{\E}\bigg[\int_\Q \big(\bfS(\ep(\underline{\bfv}^m))-\bfS(\ep(\underline{\bfv}^k))\big):\psi_1(|\underline{\bfu}^{m,k}|)\ep(\underline{\bfu}^{m,k})\dxs\bigg]\longrightarrow0,\quad m,k\rightarrow0.
\end{align*}
Due to (\ref{eq:compactpi}) we infer
\begin{align}\label{eq:monop1b}
\underline{\E}\bigg[\int_\Q \big(\bfS(\ep(\underline{\bfv}^m))-\bfS(\ep(\underline{\bfv}^k))\big):\psi_1(|\underline{\bfu}^{m,k}|)\ep(\underline{\bfv}^{m,k})\dxs\bigg]\longrightarrow0,\quad m,k\rightarrow0,
\end{align}
For $\theta\in(0,1)$ we obain
\begin{align*}
\underline{\E}\bigg[&\int_\Q \Big(\big(\bfS(\ep(\underline{\bfv}^{m}))-\bfS(\ep(\underline{\bfv}^k))\big):\ep(\underline{\bfv}^{m,k})\Big)^\theta\dxs\bigg]\\&=\int_{\Omega\times \Q} \chi_{\set{|\underline{\bfu}^{m,k}|> 1}}\Big(\big(\bfS(\ep(\underline{\bfv}^{m}))-\bfS(\ep(\underline{\bfv}^k))\big):\ep(\underline{\bfv}^{m,k})\Big)^\theta\dxs\,\mathrm{d}\p\\
&+\int_{\Omega\times \Q} \chi_{\set{|\underline{\bfu}^{m,k}|\leq 1}}\Big(\big(\bfS(\ep(\underline{\bfv}^m))-\bfS(\ep(\underline{\bfv}^k))\big):\ep(\underline{\bfv}^{m,k})\Big)^\theta\dxs\,\mathrm{d}\p\\
&=:(A)+(B).
\end{align*}
By (\ref{eq:convu1}) and (\ref{eq:convu2}) there holds
\begin{align*}
(A)&\leq \p\otimes \LL^{d+1}\big([|\underline{\bfu}^{m,k}|\geq 1]\big)^{1-\theta}\bigg(\int_{\Omega\times \Q}\big(\bfS(\ep(\underline{\bfv}^m))-\bfS(\ep(\underline{\bfv}^k))\big):\ep(\underline{\bfu}^{m,k})\dxs\,\mathrm{d}\p\bigg)^\theta\\
&\leq c\bigg(\underline{\E}\bigg[\int_\Q |\underline{\bfu}^m-\underline{\bfu}^k|^2\dxs\bigg]\bigg)^{1-\theta}\longrightarrow0,\quad m,k\rightarrow0,
\end{align*}
where we took into account H\"older's inequality.  Since $(B)$ also vanishes for $m,k\rightarrow0$ by (\ref{eq:monop1b}) we
finally have shown
\begin{align*}
\underline{\E}\bigg[\int_\Q \Big(\big(\bfS(\ep(\underline{\bfv}^m))-\bfS(\ep(\underline{\bfv}^k))\big):\ep\big(\underline{\bfv}^m-\underline{\bfv}^k\big)\Big)^\theta\dxs\bigg]
&\longrightarrow0,\quad m,k\rightarrow0,
\end{align*}
for all $\theta<1$. The monotonicity of $\bfS$ implies that $\ep(\underline{\bfv}^m)$ is Cauchy sequence $\underline{\p}\otimes \LL^{d+1}$-a.e. The limit function therefore exists $\underline{\p}\otimes \LL^{d+1}$ but has to be equal to $\ep(\underline{\bfv})$ on account of (\ref{eq:convu})$_1$.
This justifies the limit procedure in the energy integral, e.g. $\underline{\tilde{\bfS}}=\bfS(\ep(\underline{\bfv}))$ is shown and the proof of Theorem \ref{thm:main} is therefore complete.

\section{Appendix: It\^{o}'s formula in infinite dimensions}
In this section we establish a version of It\^{o}'s formula which holds for weak solutions of SPDE's on a probability space $(\Omega,\F,\p)$. Let $\bfu\in \mathcal V_{p,q}$ with $p,q\in(1,\infty)$ be a solution to the system
\begin{align}\label{eq:app1}
\begin{aligned}
\int_G\bfu(t)\cdot\bfphi\dx&=\int_G\bfu_0\cdot\bfvarphi\dx+\int_0^t\int_G\bfH:\nabla\bfphi\dxs\\&+\int_0^t\int_G \bfh\cdot\bfphi\dxt+\int_G\int_0^t\bfphi\cdot\Phi\,\dd\bfW_\sigma\dx
\end{aligned}
\end{align}
for all $\varphi\in C^\infty_0(G)$, where $\bfW$ is given by (\ref{eq:W}). We assume
\begin{enumerate}[label={\rm (I\arabic{*})}, leftmargin=*]
\item\label{itm:I1} $\bfu_0\in L^2(\Omega,\F_0,\p;L^2(G))$;
\item\label{itm:I2} $\bfH\in L^{p'}(\Omega,\F,\p;L^{p'}(\Q))$ adapted to $(\mathcal F_t)$;
\item\label{itm:I3} $\bfh\in L^{q'}(\Omega,\F,\p;L^{q'}(\Q))$ adapted to $(\mathcal F_t)$;
\item\label{itm:I4} $\Phi\in L^2(\Omega,\F,\p; L^2(0,T;L_2(U,L^2(G))))$ progressively measurable.
\end{enumerate}

\begin{lemma}[It\^{o}'s Lemma]
\label{lems:Ito}
Let $$f:L^2(G)\rightarrow \R,\quad \bfv\mapsto\int_G F(x,\bfv)\dx,$$ where $F\in C^2(G\times\R^d)$ has the following properties: for all $(x,\bfxi)\in G\times\R^d$ we have
\begin{itemize}
\item $|D_\bfxi F(x,\bfxi)|\leq c(1+|\bfxi|)$ and $|D^2_\bfxi(x,\bfxi)|\leq c$;
\item  $|D_x D_\bfxi F(x,\bfxi)|\leq c(1+|\bfxi|)$.
\end{itemize}
Let $\bfu\in \mathcal V_{p,q}$ with $p,q\in(1,\infty)$ be a solution to (\ref{eq:app1}) under \ref{itm:I1}-\ref{itm:I4}. Then there holds
\begin{align*}
f(\bfu(t))=&f(\bfu(0))+\int_0^t f'(\bfu)\,\dd\bfu+\frac{1}{2}\int_0^t f''(\bfu)\,\dd\langle\bfu\rangle_\sigma,
\end{align*}
where the term $f'(\underline{\bfu})\,\dd\underline{\bfu}$ has to be understood as an appropriate duality product.
\end{lemma}
\begin{proof}
We follow the ideas of \cite{DHV}, Prop. 1. We replace $\bfvarphi$ with the mollification $\bfvarphi_\varrho$, where $\varrho<\mathrm{dist}(\mathrm{spt}(\bfvarphi),\partial G)$. This yields
\begin{align*}
\bfu_{\varrho}(t)=(\bfu_0)_\varrho-\int_0^t\Div\bfH_\varrho\ds+\int_0^t\bfh_\varrho\ds+\sum_k\int_0^t(\Phi_k)_\rho\,\dd\beta_k,
\end{align*}
where $\Phi_k:=\Phi\bfe_k$ a.e. on $G$. So we can apply the common finite-dimensional It\^{o} formula to the real-valued process $t\mapsto f(\bfu_\varrho(t))$ and gain
\begin{align}\label{191}
f(\bfu_\varrho(t))=&f(\bfu_\varrho(0))+\int_0^t f'(\bfu_\varrho)\,\dd\bfu_\varrho+\frac{1}{2}\int_0^t f''(\bfu_\varrho)\,\dd\langle\bfu_\varrho\rangle_\sigma.
\end{align}
Here we have for $\bfw\in L^2(G)$
\begin{align*}
f'(\bfw)\bfvarphi&=\int_G D_\bfxi F(x,\bfw)\cdot\bfvarphi\dx,\quad \bfvarphi\in L^2(G),\\ f''(\bfw)(\bfvarphi,\bfpsi)&=\int_G D^2_\bfxi F(x,\bfw)(\bfvarphi,\bfpsi)\dx,\quad \bfvarphi,\bfpsi\in L^2(G).
\end{align*}
The process $t\mapsto\langle\bfu_\varrho\rangle_t$ is the quadratic variation of $\bfu_\rho$ with values in $\mathcal N(L^2(G))$(the set of nuclear operators on $L^2(G)$).\\
As a consequence of the convergence properties of the convolution a passage to the limit in \eqref{191} implies the claim (see \cite{DHV} for more details).
\end{proof}

\subsection*{Acknowledgement}
\begin{itemize}
\item The work of the author was supported by Leopoldina (German National Academy of Science).
\item The author wishes to thank M. Hofmanov\'{a} for many helpful discussions about stochastic PDEs.
\item The author is also grateful to the referee for his careful reading of the paper, and for his valuables suggestions.
\end{itemize}

\end{document}